\documentclass[12pt]{amsart}
\usepackage{t1enc}
\usepackage[latin1]{inputenc}
\usepackage[english]{babel}
\usepackage{amssymb}
\usepackage{amsmath}
\usepackage{graphics}
\usepackage{hyperref}
\usepackage{enumerate}
\usepackage{caption}
\usepackage{tikz}

\usepackage{amsthm, amssymb}
\usepackage{amsfonts}
\usepackage{epsfig,multicol}   
\usepackage{placeins}
\usepackage{mathrsfs}
\numberwithin{equation}{section}
\theoremstyle{theorem}
\newtheorem{theorem}{Theorem}[section]

\newtheorem{proposition}[theorem]{Proposition}
\newtheorem{lemma}[theorem]{Lemma}
\newtheorem{corollary}[theorem]{Corollary}
\theoremstyle{definition}
\newtheorem{definition}[theorem]{Definition}
\newtheorem{example}[theorem]{Example}
\newtheorem{remark}[theorem]{Remark}

\newcommand{\arxiv}[1]{\href{http://arxiv.org/abs/#1}{\tt arXiv:\nolinkurl{#1}}}
\newcommand{\doi}[1]{\href{http://dx.doi.org/#1}{{\tt DOI:#1}}}

\newcommand{\googlebooks}[1]{(preview at \href{http://books.google.com/books?id=#1}{google books})}

\def\<{\langle}
\def\>{\rangle}

\begin{document}

\def\hpic #1 #2 {\mbox{$\begin{array}[c]{l} \epsfig{file=#1,height=#2}
\end{array}$}}
 
\def\vpic #1 #2 {\mbox{$\begin{array}[c]{l} \epsfig{file=#1,width=#2}
\end{array}$}}

\title{Fusion categories associated to subfactors with index $3+\sqrt{5}$}
\author{Pinhas~Grossman
}
%
%

\maketitle

\begin{abstract}
We classify fusion categories which are Morita equivalent to even parts of subfactors with index $3+\sqrt{5} $, and module categories over these fusion categories. For the fusion category $\mathcal{C} $ which is the even  part of the self-dual $3^{\mathbb{Z}/2\mathbb{Z} \times \mathbb{Z}/2\mathbb{Z} } $ subfactor, we show that there are $30$ simple module categories over $ \mathcal{C}$; there are no other fusion categories in the Morita equivalence class; and the order of the Brauer-Picard group is $360$. The proof proceeds indirectly by first describing the Brauer-Picard groupoid of a $ \mathbb{Z}/3\mathbb{Z} $-equivariantization $\mathcal{C}^{\mathbb{Z}/3\mathbb{Z} } $ (which is the even part of the $4442$ subfactor). We show that that there are exactly three other fusion categories in the Morita equivalence class of $\mathcal{C}^{\mathbb{Z}/3\mathbb{Z} } $, which are all $ \mathbb{Z}/3\mathbb{Z} $-graded extensions of  $\mathcal{C} $. Each of these fusion categories admits $20$ simple module categories, and their Brauer-Picard group is $\mathcal{S}_3 $. We also show that there are exactly
five fusion categories in the Morita equivalence class of the even parts of the $3^{\mathbb{Z}/4\mathbb{Z}  }$ subfactor; each admits $7$ simple module categories; and the Brauer-Picard group is $\mathbb{Z}/2\mathbb{Z} $.
  \end{abstract}

\section{Introduction}
The classification of small-index subfactors has revealed a number of interesting examples of tensor categories which have not appeared in the representation theory of groups or quantum groups. In recent years there has been considerable attention focused on subfactors with index $3+
\sqrt{5} $, which is the first admissible composite index value above $4$. Classification of (the standard invariants of) subfactors with index $3+
\sqrt{5} $ and which have an intermediate subfactor was achieved in \cite{MR3345186}. The complete classification of subfactors with index $3+ \sqrt{5} $ was achieved in \cite{1509.00038}, where it is shown that there are exactly seven finite depth subfactor planar algebras at index $3+\sqrt{5}$, up to duality. These subfactors had been previously constructed by several authors \cite{MR3314808,1609.07604}.

In this paper we study the fusion categories which are the even parts of these subfactors. The principal even part $ \mathcal{N}$ of a finite depth subfactor $ N \subseteq M$ is the fusion category of $N$-$N $ bimodules which is tensor generated by the bimodule ${}_N M {}_N$. The object $A={}_N M {}_N$ has the structure of an algebra inside $ \mathcal{N}$. The fusion category $\mathcal{M}$ of $A$-$A$ bimodules in $\mathcal{N} $ is called the dual even part of the subfactor, and is said to be Morita equivalent to $\mathcal{N} $. 

Given a finite depth subfactor, we then have a pair of fusion categories and a Morita equivalence between them. It is natural to ask: what are all of the fusion categories in the Morita equivalence class, and what are all of the Morita equivalences between them? This information is contained in the Brauer-Picard groupoid, introduced in \cite{MR2677836}. Answering these questions can be helpful for understanding the structure of the known subfactor, and may also reveal interesting related subfactors. For example, an analysis of the Brauer-Picard groupoid of the Asaeda-Haagerup subfactor led to the discovery of a new quadratic fusion category associated to the group $\mathbb{Z}/4\mathbb{Z}  \times \mathbb{Z}/2\mathbb{Z}$. An alternative construction of  the Asaeda-Haagerup subfactor starting from this quadratic fusion category allowed for the computation of its Drinfeld center as well as the resolution of other open problems \cite{AHcat}.

The finite depth subfactors with index $3+\sqrt{5}$ have even parts which fall into six Morita equivalence classes. For two of these Morita equivalence classes the Brauer-Picard groupoids are easy to work out. Another two correspond to the unique Haagerup-Izumi subfactors (also called $3^G$ subfactors) for the two groups of order four. The final two are related to the Haagerup-Izumi subfactors for order four groups through equivariantization, as described in \cite{1609.07604}.

We first consider the even part $\mathcal{C} $ of the self-dual $3^{\mathbb{Z}/2\mathbb{Z} \times \mathbb{Z}/2\mathbb{Z} } $ subfactor. The category $\mathcal{C} $ contains $\text{Vec}_{\mathbb{Z}/2\mathbb{Z} \times \mathbb{Z}/2\mathbb{Z}} $ (the category of $\mathbb{Z}/2\mathbb{Z} \times \mathbb{Z}/2\mathbb{Z}$-graded vector spaces)  as a tensor subcategory. The simple objects are labeled by $ \alpha_g $ and $\alpha_g \rho $, where the $ \alpha_g $ represent the simple objects in  $\text{Vec}_{\mathbb{Z}/2\mathbb{Z} \times \mathbb{Z}/2\mathbb{Z}} $ and $\rho $ is a noninvertible simple object. There are algebra structures for the objects $1+ \alpha_g \rho $ for each $g $; the subfactor corresponding to each of these algebras is  the $3^{\mathbb{Z}/2\mathbb{Z} \times \mathbb{Z}/2\mathbb{Z} } $ subfactor.

The modular data of the Drinfeld center of $\mathcal{C} $ was computed in \cite{1501.07679}. The center factors as a tensor product of a rank $4$ and a rank $10$ modular tensor category. A construction of $\mathcal{C} $ coming from a conformal inclusion was obtained in \cite{FXpaper}. This construction provides an alternate description of the Drinfeld center.

It was shown in \cite{1609.07604} that there is a $\mathbb{Z}/3\mathbb{Z} $-action on $\mathcal{C} $ which cyclically permutes the nontrivial simple objects in $\text{Vec}_{\mathbb{Z}/2\mathbb{Z} \times \mathbb{Z}/2\mathbb{Z}} $. The equivariantization $\mathcal{C} ^{\mathbb{Z}/3\mathbb{Z}}$ with respect to this action contains as a subcategory $\text{Rep}_{\mathcal{A}_4} $ (the category of finite dimensional representations of the alternating group on four letters). The category $\mathcal{C} ^{\mathbb{Z}/3\mathbb{Z}}$  is the even part of the $4442$ subfactor, first constructed in \cite{MR3314808}. 

Because $ \mathcal{C}$ has many module categories and a large degree of symmetry, it is difficult to analyze its Brauer-Picard groupoid directly. Instead, we first look at the Brauer-Picard groupoid of the equivariantization $\mathcal{C} ^{\mathbb{Z}/3\mathbb{Z}}$. The category $\mathcal{C} ^{\mathbb{Z}/3\mathbb{Z}}$ is Morita equivalent to the crossed-product category
$\mathcal{C}  \rtimes {\mathbb{Z}/3\mathbb{Z}}$, which is a quasi-trivial $\mathbb{Z}/3\mathbb{Z} $-graded extension of $\mathcal{C} $. 

We show that there are exactly two other fusion categories in the Morita equivalence class of  $\mathcal{C} ^{\mathbb{Z}/3\mathbb{Z}}$, which are both also $\mathbb{Z}/3\mathbb{Z} $-graded extensions of $\mathcal{C} $, although not quasi-trivial. One of these is the category of bimodules for a $2$-dimensional algebra $ A=1+\alpha_g$ in $\text{Vec}_{\mathbb{Z}/2\mathbb{Z} \times \mathbb{Z}/2\mathbb{Z}} \subset \text{Vec}_{\mathcal{A}_4}  \subset  \mathcal{C} \rtimes \mathbb{Z}/3\mathbb{Z}$, for $0 \neq g \in \mathbb{Z}/2\mathbb{Z} \times \mathbb{Z}/2\mathbb{Z}$; this category contains as a subcategory the category of $A$-$A $ bimodules in $\text{Vec}_{\mathcal{A}_4}$, which is not equivalent to either $ \text{Vec}_{\mathcal{A}_4}$ or $ \text{Rep}_{\mathcal{A}_4}$. 

The fourth fusion category in the Morita equivalence class is more subtle to distinguish. It can be realized as the category of $ A$-$A$ bimodules for the algebra $A=1+\alpha_g \rho $ in $\mathcal{C} \rtimes \mathbb{Z}/3\mathbb{Z}$, where $0 \neq g  \in \mathbb{Z}/2\mathbb{Z} \times \mathbb{Z}/2\mathbb{Z} $. While the category of $A$-$A$-bimodules in $\mathcal{C} $ is equivalent to $\mathcal{C} $ for all $A=1+\alpha_g \rho$, it turns out that in the larger category
$\mathcal{C} \rtimes \mathbb{Z}/3\mathbb{Z}$, only the algebra $A=1+\rho$ gives a Morita autoequivalence while the algebras $1+\alpha_g \rho$ for $g \neq 0 $ do not.

The Brauer-Picard group of Morita autoequivalences of $\mathcal{C} \rtimes \mathbb{Z}/3\mathbb{Z}$ is generated by the autoequivalence coming from $1+\rho$ (which has order $2$) and an autoequivalence coming from an algebra structure for $\sum_{g \in \mathbb{Z}/2\mathbb{Z} \times \mathbb{Z}/2\mathbb{Z}} \limits \alpha_g $ (there are two different algebra structures for this object, corresponding to elements of the Schur multiplier of $\mathbb{Z}/2\mathbb{Z} \times \mathbb{Z}/2\mathbb{Z}$ , which each give Morita 
autoequivalences of order $3$).
\begin{figure}
\begin{centering}
\begin{tikzpicture}

\draw[<-,thick] (0,0.4) --(0,3.6); 
\draw[->,thick] (0.9,4) --(4.3,4);
\draw[->,thick] (0.5,3.65) --(4.6,0.5);
\node at (0,4) {$\mathcal{C} \rtimes {\mathbb{Z}/3\mathbb{Z}}    $};
\node at (4.1,2.2) {$1+\alpha_g \rho $ \ ($g \neq 0$) };
\node at (0,0) {$\mathcal{C}_3$};
\node at (5,0) {$\mathcal{C}_4$};
\node at (5,4) {$\mathcal{C} ^{\mathbb{Z}/3\mathbb{Z}}$};
\node at (2.4,4.4) {$\sum_{g \in \mathbb{Z}/3\mathbb{Z}} \limits \alpha_g $};
\node at (-0.85,2.1) {$\sum_{g \in \mathbb{Z}/2\mathbb{Z}} \limits \alpha_g    $};
\draw[->,thick] (0,4.4) arc (-70:250:0.6); 
\draw[->,thick] (-1,4.1) arc (10:340:0.6); 
\node at (-3.6,4) {$ \sum_{g \in \mathbb{Z}/2\mathbb{Z}  \times  \mathbb{Z}/2\mathbb{Z}  } \limits \alpha_g $};
\node at (-1,5.8) {$1+\rho$};
\end{tikzpicture}
\caption{``Generators'' for the Brauer-Picard groupoid of  $\mathcal{C} ^{\mathbb{Z}/3\mathbb{Z}}$}
\label{smallind}
\end{centering}
\end{figure}
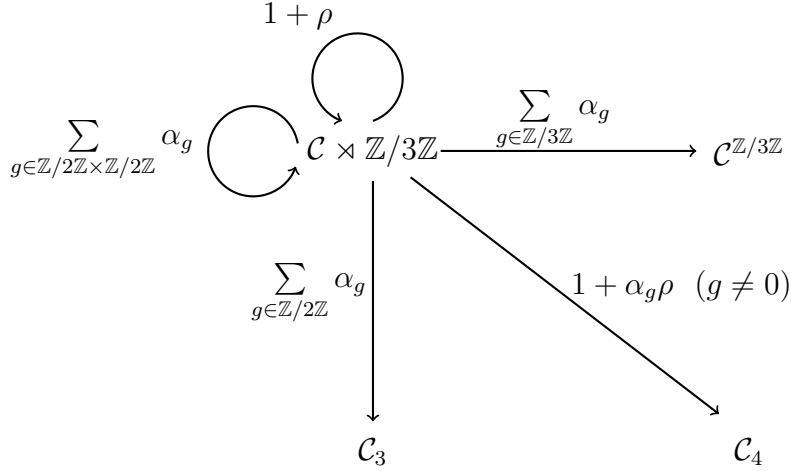

Putting all this together, we obtain our first main result.
\begin{theorem}
There are exactly four fusion categories in the Morita equivalence class of $\mathbb{C}^{\mathbb{Z}/3\mathbb{Z}} $. The Brauer-Picard group is $\mathcal{S}_3 $. 
\end{theorem}

Once we have described the Brauer-Picard groupoid of $\mathcal{C} ^{\mathbb{Z}/3\mathbb{Z}}$, we can exploit its combinatorial structure to deduce a lot of information about module categories over $\mathcal{C}$, since  $\mathcal{C}$ is realized as the trivial component of three of the fusion categories in the groupoid with respect to their $ \mathbb{Z}/3\mathbb{Z}$-gradings. In particular, we immediately see that the dual category of every simple $\mathcal{C} $-module category is again equivalent to $ \mathcal{C}$. The outer automorphism group of $\mathcal{C} $ has order $12$,  which means that each simple module category corresponds to $12$ different bimodule categories. This leads to the following result. 
\begin{theorem}
There are exactly $30$ simple module categories over $ \mathcal{C}$. The order of the Brauer-Picard group of $\mathcal{C}$ is $360$. 
\end{theorem}
The size of the Brauer-Picard group here is striking, and reflects the remarkable degree of symmetry of the Haagerup-Izumi subfactor for $\mathbb{Z}/2\mathbb{Z} \times \mathbb{Z}/2\mathbb{Z}$. The situation should be contrasted with the case of cyclic groups. For the even parts of the $3^{\mathbb{Z}/3\mathbb{Z}} $ and $3^{\mathbb{Z}/4\mathbb{Z}} $ subfactors, the orders of the Brauer-Picard groups are just $1$ (\cite{MR2909758}) and $2$ (see below), respectively.

It is difficult to work out multiplicative relations in the Brauer-Picard group directly from the list of module categories and automorphisms. In forthcoming joint work with Feng Xu, we will determine the group structure by considering the alternative description of the Brauer-Picard group as the group of braided tensor autoequivalences of the Drinfeld center \cite{MR2677836} and its connections to conformal field theory.

We next consider the even parts of the   $3^{\mathbb{Z}/4\mathbb{Z} }$ subfactor. Unlike in the $\mathbb{Z}/2\mathbb{Z}\times \mathbb{Z}/2\mathbb{Z}$ case, here the two even parts are not equivalent. It is shown in \cite{1308.5723} that the dual even part is one of the even parts of the third ``fish'' subfactor. There is therefore a third fusion category in the Morita equivalence class (the other even part of the third fish). It is shown in \cite{1609.07604} that the $2D2$ subfactor is a $\mathbb{Z}/2\mathbb{Z} $ de-equivariantization of the principal even part of the $3^{\mathbb{Z}/4\mathbb{Z} }$ subfactor, so there is also a corresponding $\mathbb{Z}/2\mathbb{Z} $-crossed product of the principal even part of the $2D2$ subfactor, which gives a fourth Morita equivalent category. Finally, a fifth Morita equivalent category can be obtained as the category of bimodules over the group algebra $A=\sum_{g \in \mathbb{Z}/4\mathbb{Z} } \limits \alpha_g$. We work out the way these categories  fit together into invertible bimodules to obtain the following result.
\begin{theorem}
There are exactly five fusion categories in the Morita equivalence class of the even parts of the $3^{\mathbb{Z}/4\mathbb{Z} }$ subfactor. The Brauer-Picard group is $ \mathbb{Z}/2\mathbb{Z} $.
\end{theorem}

\begin{figure}
\begin{centering}
\begin{tikzpicture}

\node at (0,0) {$\mathcal{P}_1    $};
\node at (3,0) {$\mathcal{P}_2    $};
\node at (0,3) {$\mathcal{P}_3    $};
\node at (-3,3) {$\mathcal{P}_5    $};
\node at (3,3) {$\mathcal{P}_4    $};
\node at (-1.5,3.40) {$\mathbb{Z}/2\mathbb{Z}$  };
\node at (-2.4,1.2) {$\mathbb{Z}/4\mathbb{Z}$  };
\draw[<->] (0.25,0.08) --(2.6,0.08);
\draw[<->] (-2.8,2.6) --(-0.4,0.3); 
\draw[<->] (-2.75,3.08) --(-0.3,3.08); 
\draw[<->] (3.1,0.28) --(3.1,2.7); 
\draw[<->] (0.25,3.08) --(2.6,3.08); 
\node at (1.5,-0.3) {$3^{\mathbb{Z}/4\mathbb{Z}}  $};
\node at (5.0,1.5) {``3rd fish'' (\&  $\mathbb{Z}/2\mathbb{Z}$)};
\node at (0.7,1.5) { $\mathbb{Z}/2\mathbb{Z}$  };
\node at (1.5,3.5) {2D2};
\draw[<->] (0.1,0.28) --(0.1,2.7); 
\end{tikzpicture}
\caption{Schematic diagram of some small index subfactors in the Brauer-Picard groupoid of $3^{\mathbb{Z}/4\mathbb{Z} } $}
\label{smallind}
\end{centering}
\end{figure}
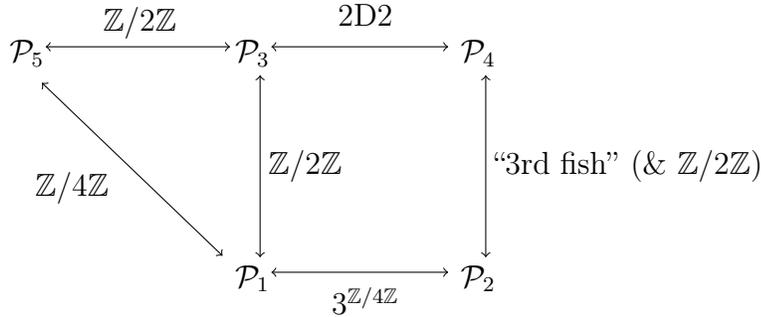

Three of the five fusion categories in the Morita equivalence class admit outer automorphisms, and the other two do not. Therefore there are exactly seven simple module categories over each of these fusion categories.

We use the combinatorial techniques developed in \cite{MR3449240} to analyze the groupoid structure of invertible bimodule categories over fusion categories. A key new feature here is the graded structure of the three other fusion categories in the Morita equivalence class of $\mathcal{C}^{\mathbb{Z}/3\mathbb{Z}}$. The $\mathbb{Z}/3\mathbb{Z}$-grading allows us to relate many module categories to algebras in the $0$-graded component $\mathcal{C} $, and in turn allows us to deduce corresponding information about module categories over $\mathcal{C} $. 

We use a computer to perform the combinatorial search for fusion modules over the Grothendieck rings of certain fusion categories, as in \cite{MR3449240}. However, the only properties of the fusion modules that we use are the collections of algebra objects that are associated to realizations of these fusion modules by module categories. The lists of such algebra objects for the fusion modules over various rings are summarized in tables in this article (Figures \ref{calgs}, \ref{c1algs}, and \ref{z4algs} below), so the full structure data of the fusion modules are not necessary to follow any of the arguments. However for the interested reader we include the full module data (which are represented as lists of non-negative integer matrices) in accompanying text files with filenames \textit{Modules\_*}. These text files are available in the arxiv source. We also use a computer to find the Grothendieck ring of the dual category of a certain module category in the proof of Lemma \ref{undual}, which is again a straightforward combinatorial calculation. 

When fitting the different fusion modules and bimodules together into a groupoid structure, we do all the multiplicative compatibility checks by hand, both to illustrate the ideas of the argument and to avoid referring to complicated computer-generated tables; however some of these calculations can be automated as well as in \cite{MR3449240}.

The paper is organized as follows. 

In Section 2, we review necessary material on subfactors and fusion categories, and develop some properties of graded module categories over graded fusion categories.

In Section 3, we describe the Brauer-Picard groupoid of the $\mathbb{Z}/3\mathbb{Z}$-equivariantization of the even part $\mathcal{C} $ of the $3^{\mathbb{Z}/2\mathbb{Z}\times \mathbb{Z}/2\mathbb{Z}}$ subfactor and use this to classify module categories over $\mathcal{C} $.

In Section 4, we describe the Brauer-Picard groupoid of the $3^{\mathbb{Z}/4\mathbb{Z}} $ subfactor. 

Section 5 is a brief comment on the other subfactors with index $3+\sqrt{5}$, included for completeness.

\subsection{Acknowledgements}
This paper is part of an ongoing program, initiated with Noah Snyder in \cite{MR2909758, MR3449240}, which aims to understand the structure of small-index subfactors by studying the representation theory of associated fusion categories. I would like to thank Noah Snyder for many helpful conversations on automorphisms of fusion categories and other topics. I would like to thank Masaki Izumi for providing me with an early version of his notes for \cite{1609.07604}, and for many helpful conversations. I would like to thank Marcel Bischoff for pointing out the appearance of the $3^{\mathbb{Z}/2\mathbb{Z} \times \mathbb{Z}/2\mathbb{Z} } $ subfactor in conformal field theory in \cite{FXpaper}, and I would like to thank Feng Xu for conversations about its Brauer-Picard group.
I am grateful to the University of Rome, Tor Vergata, for its kind hospitality during the Fall of 2016, when this paper was completed. This work was partially supported by ARC grant DP140100732 and by a UNSW Faculty of Science Silverstar Award. 
  
\section{Preliminaries}
\subsection{Fusion categories, module categories, and algebras }
\begin{definition}
\cite{MR2183279} A fusion category over an algebraically closed field $k$ is a $k $-linear semisimple rigid monoidal category with finitely many simple objects and finite-dimensional morphism spaces, such that the identity object is simple.
\end{definition}
In this paper we will always assume that $k$ is $\mathbb{C} $, the field of complex numbers. 
A unitary fusion category is a fusion category equipped with an contravariant antilinear involutive endofunctor, denoted by $*$, which fixes objects, commutes with the tensor product on morphisms, and satisfies $$ f^* \circ f=0  \text{ iff } f=0 $$
for all morphisms $f$.
Module categories over a fusion category $\mathcal{C} $ and bimodule categories between a pair of fusion categories $\mathcal{C} $ and $\mathcal{D} $ are defined in a natural way \cite{MR1976459}. 
In this paper all module categories are assumed to be semisimple. A module category is said to be simple if it is indecomposable.

One can take a relative tensor product of two bimodule categories ${}_{\mathcal{C}} {\mathcal{M}} {}_{\mathcal{D}} $ and ${}_{\mathcal{D}} {\mathcal{L}} {}_{\mathcal{E}} $ over $\mathcal{D} $ to obtain a new bimodule category ${}_{\mathcal{C}} {\mathcal{M}} \boxtimes_{\mathcal{D}} \mathcal{L} {}_{\mathcal{E}} $ \cite{MR2677836}.
A bimodule category ${}_{\mathcal{C}} {\mathcal{M}} {}_{\mathcal{D}} $ is said to be invertible if
$${}_{\mathcal{C}} {\mathcal{M}} \boxtimes_{\mathcal{D}} \mathcal{M}^{op} {}_{\mathcal{C}}\cong {}_{\mathcal{C}}  {\mathcal{C}} {}_{\mathcal{C}}
\text{ and }
{}_{\mathcal{D}} {\mathcal{M}^{op} } \boxtimes_{\mathcal{C}} \mathcal{M} {}_{\mathcal{D}}\cong {}_{\mathcal{D}}  {\mathcal{D}} {}_{\mathcal{D}}.$$ An invertible bimodule category is also called a Morita equivalence, and $\mathcal{C} $ and $ \mathcal{D}$ are said to be Morita equivalent.

The dual category $\mathcal{D}=({}_{\mathcal{C}} {\mathcal{M}} )^*$  of a left module category ${}_{\mathcal{C}} \mathcal{M} $ over a fusion category $\mathcal{C} $ is the category of module endofunctors; the category $\mathcal{M} $ is a right module category over $\mathcal{D}$. If ${}_{\mathcal{C} }\mathcal{M} $ is a simple module category, the bimodule category ${}_{\mathcal{C}} {\mathcal{M}} {}_{\mathcal{D}} $ is invertible. Conversely, if  ${}_{\mathcal{C}} {\mathcal{M}} {}_{\mathcal{D}} $ is any invertible bimodule category, then $\mathcal{D} $ is equivalent to $({}_{\mathcal{C}} {\mathcal{M}} )^*$. Thus two fusion categories are Morita equivalent if and only if each is equivalent to the dual category of a simple module category over the other.

The outer autmorphism group $\text{Out}(\mathcal{C}) $ of a fusion category $\mathcal{C} $ is the quotient of the group of tensor autoequivalences of $\mathcal{C} $ (considered modulo monoidal natural isomorphism) by the subgroup of inner autoequivalences (conjugation by invertible objects).  If ${}_{\mathcal{C}} \mathcal{M} $ is a module category and $\mathcal{D}=({}_{\mathcal{C}} {\mathcal{M}} )^*$, then the invertible $\mathcal{C} $-$\mathcal{D} $ bimodule categories which extend ${}_{\mathcal{C}} \mathcal{M} $ are parametrized by $\text{Out}(\mathcal{D}) $ (see \cite{1407.2783}).
 
 \begin{example}
 Let $G$ be a finite group, and consider $\text{Vec}_G$, the fusion category of $G$-graded finite dimensional vector spaces. Then $\text{Out}(\text{Vec}_G) \cong H^2(G,\mathbb{C}^*) \rtimes \text{Out}(G) $, where $\text{Out}(G)$ is the outer automorphism group of $G$.
 \end{example}
 
Another way to understand module categories is through the notion of an algebra in a tensor category.

\begin{definition}
An algebra in a tensor category $\mathcal{C} $ is an object $A$ together with a unit morphism $1 \rightarrow A $ and a multiplication morphism $A \otimes A \rightarrow A $ satisfying the usual unit and associativity relations.
\end{definition}

One can define a module over an algebra in a tensor category in an obvious way. If $A$ is an algebra in a fusion category $\mathcal{C} $, the category of (right) $ A$-modules in $\mathcal{C} $ is a (left) module category over $\mathcal{C} $, although not necessarily semisimple. An algebra is called simple if its category of modules is semisimple and indecomposable as a $\mathcal{C} $-module category.

\begin{definition}
A division algebra $A$ in a fusion category is a simple algebra $A$ such that $A $ is simple as a left $A$-module.
\end{definition}
A division algebra (in a fusion category over a field with characteristic $0$) has a canonical Frobenius algebra structure \cite{MR3449240}. 

For a simple module category ${}_{\mathcal{C}} {\mathcal{M}} $, there is an internal hom bifunctor from $\mathcal{M} \times \mathcal{M} $ to $\mathcal{C} $, $$(M_1,M_2) \mapsto \underline{\text{Hom}}(M_1,M_2) .$$ We refer to $\underline{\text{End} }(M)=\underline{\text{Hom}}(M,M) $ as the internal end of the object $M$. An object in an invertible bimodule category has both a left and a right internal end.

\begin{theorem}\cite{MR1976459}
Let $M$ be a simple object in a simple module category ${}_{\mathcal{C}} {\mathcal{M}} $ over a fusion category $\mathcal{C}$. Then $\underline{\text{End} }(M)$ is a division algebra, and ${}_{\mathcal{C}} {\mathcal{M}} $ is equivalent to the category $\text{mod-}\underline{\text{End} }(M)$ of (right) $\underline{\text{End} }(M)$-modules in $\mathcal{C} $.
\end{theorem}
Thus any simple $\mathcal{C}$-module category is equivalent to $\text{mod-}A$ for some division algebra $A\in \mathcal{C}$; the dual category is then equivalent to $A\text{-mod-}A $, the category of $A $-$A$ bimodules in $ \mathcal{C}$.
\begin{example}
 For a finite group $G$,  algebras in $\text{Vec}_G$ correspond to pairs $(H,\omega)$, where $H$ is a subgroup of $G$ and $\omega  $ is an element of $H^2(H,\mathbb{C}^*) $. Module categories over $\text{Vec}_G $ are parametrized by such pairs modulo conjugation by elements of $G$ (since algebras coming from conjugate pairs give equivalent module categories). If $A$ is the algebra corresponding to $H=G$ with the trivial cocycle, then the category of $ A$-$A$ bimodules in $\text{Vec}_G$ is equivalent to $\text{Rep}_G$, the category of finite-dimensional representations of $G$.
\end{example}

\subsection{Composition of bimodule categories and the Brauer-Picard groupoid}

\begin{definition}\cite{MR2677836}
The Brauer-Picard groupoid of a fusion category $\mathcal{C} $ is a $3$-groupoid whose objects
are fusion categories in the Morita equivalence class of $\mathcal{C} $; whose $1$-morphisms are invertible bimodule categories; whose $2$-morphisms are equivalences of bimodule categories; and whose $3$-morphisms are isomorphisms between such equivalences. \end{definition}

The Brauer-Picard group of $\mathcal{C} $ is the group of Morita autoequivalences of $\mathcal{C} $ (considered up to equivalence). It is a finite group which contains as a subgroup $\text{Out}(\mathcal{C}) $, whose elements can be thought of as twists of the trivial bimodule category on one side.
The Brauer-Picard group of $\mathcal{C} $ is isomorphic to the group of braided tensor autoequivalences of the Drinfeld center $ \mathcal{Z}(\mathcal{C})$ \cite{MR2677836}.

A basic problem is: given a fusion category $\mathcal{C} $, describe its Brauer-Picard groupoid. (In this paper we will only consider the objects and $1$-morphisms.) This problem is closely related to the problems of finding all division algebras in $\mathcal{C} $ and finding all simple module categories over $\mathcal{C} $.

A useful tool is the bicategorical structure of bimodules over division algebras in $\mathcal{C} $ \cite{MR2075605}.

\begin{proposition}
Let $ \mathcal{C}$ be a fusion category  (over a field of characteristic $0$). There is a rigid bicategory whose objects are division algebras in $\mathcal{C} $, whose $1$-morphisms are bimodules between division algebras, and whose $2$-morphisms are bimodule morphisms.  
\end{proposition}
 The composition of $1$-morphisms is given by the relative tensor product over the common algebra. If $A$ and $B$ are division algebras in $\mathcal{C}$, the dual of an $A$-$B$ bimodule $M$ is the dual object $M^*$ which has the structure of a $B$-$A$ bimodule. The left internal hom of $M$ is then  $M \otimes_ B M^*$, which is an $A$-$A$ bimodule. 
 
 As a consequence of the bicategorical structure, we have Frobenius reciprocity: if $A$, $B$, and $C$ are division algebras in $\mathcal{C} $, and ${}_A K {}_B $, ${}_B L {}_C $, and ${}_A M {}_C $ are bimodules,
then $\text{Hom}( K \otimes_B L ,M ) \cong \text{Hom}(K, M \otimes_C L^*  )  $, and similarly for other permutations of $K $, $L$, and $M$.

The relationship of the bicategory of bimodules over division algebras with the Brauer-Picard groupoid is as follows. Given fusion categories $\mathcal{A} $, $\mathcal{B} $, and $\mathcal{C} $  and invertible bimodule categories ${}_{\mathcal{A}} {\mathcal{K}} {}_{\mathcal{B}} $ and ${}_{\mathcal{B}} {\mathcal{L}} {}_{\mathcal{C}} $, there are division algebras $A$, $C$ in $\mathcal{B} $, and tensor equivalences 
 from $\mathcal{X} $ to the category of $ X$-$X$ bimodules in $\mathcal{C} $ for $X \in \{ A,C\} $,
such that:
\begin{enumerate}

\item $\mathcal{K}$ is equivalent to the category of left $A$-modules in $ \mathcal{B}$ (as a right $ \mathcal{B}$-module category);
\item  $\mathcal{L} $ is equivalent to the category of right $C $-modules in $\mathcal{B} $ (as a left $\mathcal{B} $-module category);
\item $ \mathcal{K} \boxtimes_{\mathcal{B}} \mathcal{L}$ is equivalent to the category of $A$-$C $ bimodules in $\mathcal{B} $ (as an $\mathcal{A} $-$\mathcal{C} $ bimodule category).
\end{enumerate}
Thus from this translation we also have a form of Frobenius reciprocity for objects in bimodule categories in the Brauer-Picard groupoid (although there is no canonical choice for the division algebras $A$ and $C$).
\subsection{Decategorification and combinatorics}

The Grothendieck ring of a fusion category $\mathcal{C} $ is the based ring with basis indexed by the simple objects of  $ \mathcal{C}$ and multiplicative structure constants given by the tensor product decomposition rules of  $\mathcal{C} $. The type of based ring which occurs is called a fusion ring. In a similar way, every simple module category over a fusion category gives a fusion module, and every invertible bimodule category gives a fusion bimodule (see \cite{MR3449240} for definitions).

There is a unique homomorphism from the Grothendieck ring of $\mathcal{C} $ to the real numbers which is positive on basis elements, called the Frobenius-Perron dimension. The Frobenius-Perron dimension can be uniquely extended to fusion modules and bimodules in a way compatble with the fusion module structure and (decategorified) internal hom. The Frobenius-Perron dimension of objects in invertible bimodule categories is multiplicative with respect to tensor product maps coming from relative tensor products of bimodule categories.

\begin{remark} There ia also a notion of quantum dimension in a spherical fusion category. However for unitary fusion categories, which include all examples in this paper, the quantum dimension coincides with the Frobenius-Perron dimension. We will use the word dimension to mean Frobenius-Perron dimension, since it will not always be obvious a priori that all fusion categories in the Morita equivalence class are unitary.
\end{remark}

To classify module categories over a fusion category, one strategy is to first compute the fusion modules over the Grothendieck ring, and then classify the module categories which realize each fusion module. If the Grothendieck ring is ``small'', the list of all fusion modules can sometimes be found through a combinatorial search with a computer.     

To classify the module categories realizing a given fusion module $R$, one can try to classify the algebras whose categories of modules realize $R$. From the data of a given fusion module $R$, one can read off the list of objects in the fusion category $\mathcal{C}$ which must be the internal ends of the simple objects in any module category realizing $R$. One can then look for the smallest or simplest such object, and sometimes the algebra structures on these objects are known or can be classified. 

In many cases even the smallest candidate algebra object corresponding to a fusion module will be mysterious. Therefore for a given fusion ring, even if we can write down the full list of fusion modules, we may only be able to directly classify categorifications of a few fusion modules through known information about algebra structures. The next step is to exploit the combinatorial structure of the Brauer-Picard groupoid by looking at multiplicative compatibility of fusion modules and bimodules. 

The idea is that we try to simultaneously classify module categories over all the fusion categories in the Morita equivalence class, and all invertible bimodule categories between them. Any two invertible bimodule categories sharing a common fusion category can be composed in the Brauer-Picard groupoid, and Frobenius reciprocity and the Frobenius-Perron dimension give strong combinatorial constraints on the fusion bimodules involved.
%
%
%
By starting with a few known small algebras, we can sometimes deduce the entire Brauer-Picard groupoid structure by considering compositions and exploiting the combinatorial constraints coming from multiplicative compatibility of fusion modules and bimodules. This program was carried out for the Haagerup and Asaeda-Haagerup fusion categories in \cite{MR2909758, MR3449240,AHcat}.

We mention a simple example of multiplicative compatibility, which is a direct consequence of Frobenius reciprocity and the multiplicative property of Frobenius-Perron dimension. 
\begin{definition}
If $X$ is an object in a fusion category $\mathcal{C} $ 
and $K $ is a fusion module over the Grothendieck ring of $\mathcal{C}$, we say that $K$ corresponds to an algebra structure on $ X$ if a module category over $\mathcal{C} $ which realizes $K$ has a simple object whose internal end is $X$. 
\end{definition}


\begin{proposition} \label{easycomp}
Let ${}_{\mathcal{A}} \mathcal{K} {}_{\mathcal{B}} $ and ${}_{\mathcal{B}} \mathcal{L} {}_{\mathcal{C}} $ be invertible bimodule categories. Suppose $\mathcal{K} {}_{\mathcal{B}} $ corresponds to an algebra $X$ in $\mathcal{B} $ and ${}_{\mathcal{B}} \mathcal{L} $ corresponds to an algebra $Y$ in $\mathcal{B} $. If $\text{dim}( \text{Hom}(X,Y)) =1$ (as objects in $\mathcal{B} $), then ${}_{\mathcal{A}} (\mathcal{K} \boxtimes_{\mathcal{B}} \mathcal{L} ) {}_{\mathcal{C}} $ corresponds to algebras $Z_{\mathcal{A}}$ and $Z_{\mathcal{C}} $ in $\mathcal{A} $ and  $\mathcal{C} $, respectively, with dimension $\text{dim}(Z_{\mathcal{A}})=\text{dim}(Z_{\mathcal{C}})=\text{dim}(X)\text{dim}(Y) $.

More generally, if $\text{dim}( \text{Hom}(X,Y)) =n$, then ${}_{\mathcal{A}} (\mathcal{K} \boxtimes_{\mathcal{B}} \mathcal{L} ) {}_{\mathcal{C}} $ contains an object $M$ with $\text{dim}(M)^2=\text{dim}(X)\text{dim}(Y) $ and $\text{dim}( \text{Hom}(M,M) )=n $.
\end{proposition}

This very weak constraint is sometimes enough to determine the relative tensor product of two bimodule categories, since there may only be one invertible $\mathcal{A} $-$\mathcal{C} $-bimodule category whose objects have compatible dimensions. In fact, this is the only form of multiplicative compatibility we will use in this paper, and we always check this by hand (in contrast to the arguments in \cite{MR3449240}, which used more extensive computer calculations to check for multiplicative compatibility of fusion modules and bimodules).

Another useful tool is to look for subalgebras of known algebras. Let $ \mathcal{C}_0$ be a tensor subcategory of a fusion category $\mathcal{C} $, and let $A$ be a division algebra in $\mathcal{C} $. Then $A$ has a subalgebra $A_0$ given by the maximal subobject of $A$ which belongs to $\mathcal{C}_0 $.
We can then ``divide'' the module category over $A$ by a bimodule category corresponding to $ A_0$ as follows.

\begin{proposition}\label{division} \cite[Lemma 3.8]{AHcat}
Let ${}_{\mathcal{A}} \mathcal{K} {}_{\mathcal{B}} $ be an invertible bimodule category corresponding to a division algebra $A$ in $\mathcal{A}$, and let $A_0 $ be a subalgebra of $A$. Let $\mathcal{C} $ be the dual category
of ${A_0} $-mod, with corresponding bimodule category ${}_{\mathcal{C}} \mathcal{L} {}_{\mathcal{A} }$. Then ${}_{\mathcal{C}} (\mathcal{L} \boxtimes_{\mathcal{A}} \mathcal{K} ) {}_{\mathcal{B}} $ 
corresponds to a division algebra $C$ with dimension $ \text{dim}(C)=
\displaystyle \frac{\text{dim(A)}}{ \text{dim}(A_0)}$.
\end{proposition}
 If $\mathcal{A} $ in Lemma \ref{division} is unitary and $A$ is a Q-system, then $\mathcal{C} $, which is the category of $ A$-$A$ bimodules in $\mathcal{A} $, inherits a unitary structure, and then $C$ can be taken to be a Q-system as well.

Finally, in the proof of Lemma \ref{undual} we will need to find the dual fusion ring of a fusion module over a fusion ring; we use the method of \cite[Section 3.1]{AHcat}.

\subsection{Graded fusion categories}

\begin{definition} \cite{MR2677836}
Let $G$ be a finite group. A (faithful) $G$-grading on a fusion catgory $\mathcal{C} $ is a direct sum decomposition  $$\mathcal{C}=\bigoplus_{g \in G} \limits \mathcal{C}_g ,$$ where the $\mathcal{C}_g $ are nonzero full Abelian subcategories and the tensor product maps $ \mathcal{C}_g \times \mathcal{C}_h$ to $\mathcal{C}_{gh} $ for all $g,h \in G$. A $G$-extension of $\mathcal{C} $ is a faithful $G$-graded fusion category whose trivial component is equivalent to $\mathcal{C} $.
\end{definition}
Each homogeneous component of a $G$-extension of $\mathcal{C} $ is an invertible $\mathcal{C} $-$\mathcal{C} $ bimodule category \cite{MR2677836}. The extension is said to be quasi-trivial if each homogenous component contains an invertible object. This condition is the same as  requiring that each homogenous component be equivalent to the trivial left/right $\mathcal{C} $-module category (though not necessarily trivial as a $\mathcal{C} $-$\mathcal{C} $ bimodule category).

Let $\mathcal{C} $ be a $G$-graded extension of $\mathcal{C}_0 $.  If $A$ is a division algebra in $\mathcal{C}_0 $, then the (right) $\mathcal{C} $-module category $A$-mod and its dual category $A$-mod-$A$ inherit a $G$-grading from $\mathcal{C} $. Moreover, the homogeneous components of $A$-mod are module categories over $\mathcal{C}_0 $. 

\begin{proposition}\label{cinv}
Let $\mathcal{C} $ be a $G$-graded extension of $\mathcal{C}_0 $, and let $A$ be a division algebra in $\mathcal{C}_0 $, with corresponding graded module category $ A$-mod. Then the homogeneous components ($A $-mod)$_g$ and ($A $-mod)$_h$ are equivalent as $\mathcal{C}_0 $-module categories iff ($ A\text{-mod-}A)_{g^{-1}h}$ contains an invertible object.
\end{proposition}
\begin{proof}
First note that the internal end of any simple object in $A $-mod is in $\mathcal{C}_0$. Moreover, the internal end of a simple object in ($A $-mod)$_g \subset A\text{-mod}$ is the same whether thought of as an object in the $\mathcal{C}_0$-module category  ($A $-mod)$_g $ or in the larger $\mathcal{C} $-module category $A$-mod. Two simple module categories are equivalent iff they each contain a simple object with the same internal end. Two simple objects $X$ and $Y$ in a simple module category have the same internal end iff there is an invertible object $Z$ in the dual category such that $X \otimes Z\cong Y $. Therefore ($A $-mod)$_g$ is equivalent to ($A $-mod)$_h$ iff there are simple objects $X \in  (A\text{-mod})_g$  and $Y \in (A\text{-mod})_h$, and an invertible object $ Z$ in $A $-mod-$A$ (necessarily in  ($A\text{-mod-}A)_{g^{-1}h}$) such that  $X  \otimes Z = Y $. But this occurrs iff ($A \text{-mod-}A)_{g^{-1}h}$ contains an invertible object.
\end{proof}
In particular, if $A $-mod-$A$ is a quasi-trivial extension of its trivial component, then all of the homogenous components of $A $-mod are equivalent as $\mathcal{C}_0 $-module categories.
\begin{proposition}\label{gralgs}
Let $\mathcal{C} $ be a quasi-trivial $G$-graded extension of $\mathcal{C}_0 $, and let $A$ be a division algebra in $\mathcal{C}_0 $. Let $B$ be another division algebra in $\mathcal{C} $. Then $A$-mod and $B$-mod are equivalent as $\mathcal{C} $-module categories iff $B \in \mathcal{C}_0$ and there is an invertible object $Z \in \mathcal{C} $ such that the trivial components of $(Z\otimes B \otimes Z^{-1}) $-mod and $ A$-mod are equivalent as $\mathcal{C}_0 $-module categories.
\end{proposition}
\begin{proof}
If there is an invertible object $Z \in \mathcal{C}$ such that the trivial components of $(Z^{-1} \otimes B \otimes Z) $-mod and $ A$-mod are equivalent as $\mathcal{C}_0 $-module categories,  then $(Z^{-1}\otimes B \otimes Z) $-mod and $ A$-mod must correspond to a common algebra in  $\mathcal{C}_0$, so $(Z^{-1}\otimes B \otimes Z )$-mod and $ A$-mod are equivalent as $\mathcal{C}$-module categories. Therefore $ B$-mod is also equivalent to $A $-mod as a $\mathcal{C}$-module category (this direction does not require quasi-triviality). For the converse, since the extension is quasi-trivial, every simple object $X \in  (A\text{-mod})_g$ can be written as $X=Y \otimes Z $, where $Z$ is an invertible object in $\mathcal{C}_g $ and $Y$ is in  $(A\text{-mod})_0$. Therefore 
$\underline{\text{End} }(X)=\underline{\text{End} }(Y \otimes Z)=Z^{-1}\otimes \underline{\text{End} }(Y) \otimes Z$. This shows that the internal end of any simple object in $A $-mod is a conjugate of the internal end of a simple object in $(A\text{-mod})_0$. In particluar, if $B $-mod is equivalent to $A$-mod, then $B$ must be of the form $B\cong Z^{-1}\otimes \underline{\text{End} }(Y) \otimes Z$, for some invertible object $Z$ and simple $Y$ in $(A\text{-mod})_0$. Then $(Z\otimes B\otimes Z^{-1})\text{-mod}=\underline{\text{End} }(Y) $-mod and $A$-mod are equivalent as $\mathcal{C}_0 $-module categories.
\end{proof}
\begin{proposition} \label{lrends}
Let $\mathcal{C} $ be a $G$-graded extension of $\mathcal{C}_0 $, and suppose that $\mathcal{C}_0 $ is the only tensor subcategory for which $\mathcal{C} $ is a $G$-graded extension. Suppose further that the Brauer-Picard group of $\mathcal{C} $ is generated by invertible $\mathcal{C} $-$\mathcal{C} $-bimodule categories which each contain a simple object whose (right or left) internal end is in $\mathcal{C}_0 $. Then the right and left internal ends of every simple object in every invertible $\mathcal{C} $-$\mathcal{C} $-bimodule category are also in $\mathcal{C}_0 $.
\end{proposition}
\begin{proof}
Let $\mathcal{M} $ be a simple (right) $\mathcal{C} $-module category, let $M \in \mathcal{M}$ be a simple object, and let $A=\underline{\text{End}}(M) $. Then $\mathcal{M} {}_{\mathcal{C}}$ is equivalent to $A$-mod. If $A \in \mathcal{C}_0$, then the dual category $A$-mod-$A$ inherits  a $G$-grading from $\mathcal{C} $, and the left internal end of $A$ is $A\otimes A$, which is in the trivial component of $A$-mod-$A$ with respect to this grading. An invertible $\mathcal{C} $-$\mathcal{C} $-bimodule category extending  $\mathcal{M}{}_{\mathcal{C}} $ corresponds to a tensor equivalence $\phi: \mathcal{C} \rightarrow A$-mod-$A$. Since $\mathcal{C}_0 $ is the only tensor subcategory for which $\mathcal{C} $ is a $G$-graded extension, we must have $\phi^{-1}(A \otimes A) \in \mathcal{C}_0 $ for any such $\phi $. It follows that for any simple object in an invertible $\mathcal{C} $-$\mathcal{C} $-bimodule category, the left internal end is in $\mathcal{C}_0 $ iff the right internal end is in $\mathcal{C}_0 $. Moreover, by Proposition \ref{gralgs}, if an invertible $\mathcal{C} $-$\mathcal{C} $-bimodule category contains a simple object whose internal end is in $\mathcal{C}_0 $, then the same is true for all simple objects. 

So it suffices to show that the relative tensor product of two invertible $\mathcal{C} $-$\mathcal{C} $-bimodule categories which each contain a simple object whose internal ends are in $\mathcal{C}_0 $ again satisfies this property. Suppose we have two such $\mathcal{C} $-$\mathcal{C} $-bimodule categories $\mathcal{K} $ and $\mathcal{L} $. Then there are division algebras $A$ and $B$ in $\mathcal{C}_0 $ and tensor equivalences $\phi: \mathcal{C} \rightarrow{} A\text{-mod-}A$ and $\psi: \mathcal{C} \rightarrow B\text{-mod-}B$
such that $\mathcal{K} \cong A\text{-mod} $, $\mathcal{L} \cong \text{mod-B} $, and $\mathcal{K}\otimes_{\mathcal{C}} \mathcal{L}\cong A\text{-mod-}B $ (with the left and right actions of $\mathcal{C} $ on $A\text{-mod-}B $ given by $\phi $ and $\psi $). 
Then the internal end of $A\otimes B \in \mathcal{K}\otimes_{\mathcal{C}} \mathcal{L}   $ in $A $-mod-$A$ is $ (A \otimes B)\otimes_B (B \otimes A)  \cong A\otimes B \otimes A$. Since $A$ and $B$ are both in $\mathcal{C}_0 $,  the bimodule $A\otimes B \otimes A $ is in the trivial component of  $A\text{-mod-}A $ with respect to the $G$-grading inherited from $\mathcal{C} $. Since $\mathcal{C}_0 $ is the unique subcategory for which $\mathcal{C} $ is a $G$-graded extension, this implies that $\phi^{-1}(A\otimes B \otimes A)  \in \mathcal{C}_0$. Therefore the internal ends of the simple summands of $A\otimes B $ are also in $ \mathcal{C}_0$. 
\end{proof}
\begin{corollary}\label{lrends2}
Let $ \mathcal{C}$ satisfy the assumptions of the proposition, let $A$ be a division algebra in $\mathcal{C}_0 $, and let $ \mathcal{C}'=A\text{-mod-}A$. Then every simple (right) $\mathcal{C} $-module category whose dual category is equivalent to $\mathcal{C}' $ is equivalent to $B $-mod for a division algebra $B$ in $\mathcal{C}_0 $. Moreover, the left and right internal ends of every simple object in every Morita autoequivalence of $\mathcal{C}' $ are in the trivial component of $\mathcal{C}'$ with respect to the grading inherited from $\mathcal{C} $.
\end{corollary}
\begin{proof}
Every invertible $\mathcal{C}'$-$\mathcal{C} $-bimodule category is a relative tensor product of $A $-mod with a Morita autoequivalence of $\mathcal{C} $,  and is therefore equivalent to $A\text{-mod-}C$ for some division algebra $C \in \mathcal{C}_0$ (with the right action of $\mathcal{C} $ given by a tensor equivalence $\mathcal{C} \cong C\text{-mod-}C  $). Then as in the proof of the proposition, the internal end of $A \otimes C $ in $ C\text{-mod-}C  \cong \mathcal{C}$ is in the trivial component $\mathcal{C}_0 $. So we may take $B$ to be the internal end of a simple summand of $A\otimes C$.
Similarly, every invertible $\mathcal{C}'$-$\mathcal{C}' $-bimodule category is equivalent to $A\text{-mod-}C$ for some division algebra $C \in \mathcal{C}_0$ (again with the right action of $\mathcal{C}' $ given by a tensor equivalence $\mathcal{C}' \cong C\text{-mod-}C  $), so the internal end of $A \otimes C $ in $A $-mod-$A$ is in the trivial component.
\end{proof}
\subsection{Equivariantization}
Let $G$ be a finite group acting on a fusion category $\mathcal{C} $ by tensor autoequivalences, with $\phi_g $ the autoequivalence corresponding to the group element $g$.  This means that there are natural isomorphisms $F_{g,h}: \phi_g \circ \phi_h \rightarrow \phi_{gh}  $ satisfying certain coherence relations. The equivariantization $\mathcal{C}^G $ is a fusion category whose objects are pairs $(X,\{f_g\}_{g \in G} )$, where $X$ is an object in $\mathcal{C} $ and $f_g: \phi_g(X)  \rightarrow X$ are isomorphisms such that 
$$f_g \circ \phi_g(f_h)=f_{gh}\circ F_{g,h} , \quad \forall g,h \in G.$$
Morphisms and tensor products are defined in a natural way. The equivariantization should be thought of as the category of ``fixed points'' of the action of $G$ on $\mathcal{C}$, and it is Morita equivalent to the crossed product category $\mathcal{C} \rtimes G $, which is a quasi-trivial $G$-graded extension of $\mathcal{C} $ \cite{MR1815142}. The Morita equivalence corresponds to the group algebra of $G$ in $\text{Vec}_G \subseteq  \mathcal{C} \rtimes G $ (the sum of all of the simple objects), and $\mathcal{C}^G$ therefore contains a copy of $\text{Rep}_G $. 


There is an inverse construction called de-equivariantization, which recovers $\mathcal{C} $ from $\mathcal{C}^G $; see \cite{MR2609644} for details.

\subsection{Subfactors}
A subfactor in a unital inclusion $N \subseteq M $ of II$_1$ factors. The subfactor has finite index if the commutant $N'$ in the standard representation of $N \subseteq M$ on $L^2(M)$ is also a II$_1$ factor. The principal even part $\mathcal{N} $ of a finite-index subfactor is the category of $N$-$N $ bimodules tensor-generated by the bimodule ${}_N M {}_N $. The category $\mathcal{N} $ is a $\mathbb{C} $-linear semisimple rigid monoidal category, and the object ${}_N M {}_N $ is an algebra in $\mathcal{N} $. The subfactor has finite depth if $\mathcal{N} $ has finitely many simple objects, in which case $\mathcal{N} $ is a fusion category. 

The principal graph of a finite-depth subfactor is the induction-restriction graph of the object ${}_N M {}_M $
in the module category of $N$-$M $ bimodules it generates over $\mathcal{N} $; the dual graph is the induction-restriction graph of ${}_N M {}_M $ with respect to the dual category, which is the category of $M$-$M$ bimodules tensor-generated by $({}_M M {}_N)\otimes_N({}_N M {}_M)$. 

A fusion category arising as the even part of a subfactor has a unitary structure, coming from the inner product given by the trace of the II$_1$ factor.
\begin{definition}\cite{MR1257245}
A $Q$-system in a unitary fusion category is an algebra such that the multiplication map is a co-isometry.
\end{definition}

A Q-system is in particular a division algebra. If $N \subseteq M $ is a finite depth subfactor, then $ {}_N M {}_N$ has a Q-system structure in $ \mathcal{N}$. Conversely, any simple Q-system in a unitary fusion category arises in this way from a subfactor. (The relationship between subfactors and Q-systems holds more generally for finite-index subfactors and countably generated rigid C$^*$-tensor categories; see \cite{MR1257245,MR1444286}.) Every finite-index subfactor has a dual subfactor, which corresponds to the Q-system $({}_M M {}_N)\otimes_N({}_N M {}_M)$ in the dual category.

%
%

The standard invariant of a subfactor is a shaded planar algebra whose underlying vector spaces are the endomorphism spaces of alternating relative tensor products of ${}_N M {}_M $ and ${}_M M {}_N $; see \cite{math.QA/9909027}. For finite depth subfactors of the hyperfinite II$_1$ factor, the standard invariant is a complete invariant \cite{MR1055708}.

The index and standard invariant can also be defined for properly infinite factors. Let $M$ be the hyperfinite Type III$_1$ factor, and consider the category $\text{End}_0(M) $ of finite index endomorphisms of $M$, where morphisms are given by intertwiners: $$\text{Hom}(\rho,\sigma)=\{t \in M : t \rho(x)=\sigma(x) t, \ \forall x \in M \} .$$
This is a rigid C$^*$-tensor category, with tensor product given by composition:
$$\rho \otimes \sigma=\rho \circ \sigma .$$
Every unitary fusion category can be realized as a full tensor subcategory of $\text{End}_0(M) $ for the hyperfinite Type III$_1$ factor $M$. Following common practice, we often suppress the tensor product symbol and write $\rho \sigma $ for $\rho \otimes \sigma $ when thinking of a fusion category as embedded inside $\text{End}_0(M)$. Also, we often suppress ``Hom'' and write $(\rho,\sigma) $ for the intertwiner space.

The direct sum of objects in $\text{End}_0(M) $ is defined up to isomorphism. We will often represent the direct sum of objects using an ordinary ``$+ $'' symbol.

\subsection{Haagerup-Izumi subfactors}
Let $G$ be a finite Abelian group. Let $\mathcal{C} $ be a unitary fusion category which contains $\text{Vec}_G$ as a full tensor subcategory and which is tensor generated by a simple object $\rho $ satisfying $$\rho \otimes \rho\cong 1\oplus \sum_{g \in G} \limits \alpha_g \otimes \rho  \text{ and } \alpha_g \otimes \rho \cong \rho \otimes \alpha_{g^{-1}} $$ (where $\{ \alpha_g \}_{g \in G}$ are the simple objects of $\text{Vec}_G $.) 

Suppose $1\oplus \rho$ admits a Q-system, Then the corresponding subfactor is called an Haagerup-Izumi subfactor for $G$. Such subfactors are also sometimes called $3^G$ subfactors, since their principal graphs consist of a $|G|$-valent vertex, along with $|G|$ ``legs'' of length $3$ emanating from that vertex. The index of a $3^G$ subfactor is $1+d$, where $d^2=1+|G| \cdot d $.

Izumi classified $3^G$ subfactors in terms of solutions to certain polynomial equations in \cite{1609.07604}. The unique $3^G$ subfactor for $G=\mathbb{Z}/3\mathbb{Z}$ is the Haagerup subfactor \cite{MR1686551}, and it is shown in \cite{1609.07604} that there is a unique $3^G$ subfactor for each of the two order $4$ groups. (For further discussion of $3^G$ subfactors for $G$ of odd order, see \cite{MR2837122}.)

\begin{figure}
\centerline{ \includegraphics[width=0.7in]{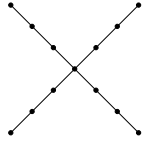}  .} 
\caption{The $3^G$ principal graph for $|G|=4$.}
\end{figure}

Izumi's construction of a $3^G$ subfactor from a solution to his polynomial equations realizes the principal even part of the subfactor explicitly as a subcategory of $\text{End}_0(M)$ for a Type III factor $M$. The von Neumann algebra $M$ is the closure of a Cuntz algebra whose generators corresponding to simple summands of $\rho^2 $. The automorphisms $\alpha_g $ act as certain signed permutations of the generators and the solutions to the polynomial equations give structure constants for the definition of the endomorphism $\rho $ on the Cuntz algebra. In this realization, the relations $\alpha_g\alpha_h= \alpha_{gh} $ and $\alpha_g \rho =\rho \alpha_{g^{-1}} $ hold exactly (not just up to isomorphism). This is useful for performing many calculations, such as constructing (de-)equivariantizations, describing dual categories and calculating automorphism groups. 

We briefly recall the orbifold constructions from \cite{1609.07604} for the specific cases which we will need.

Let $ \mathcal{C}$ be the principal even part of the $3^{\mathbb{Z}/2\mathbb{Z} \times \mathbb{Z}/2\mathbb{Z}  } $ subfactor, realized in $\text{End}_0(M) $ as above. There is an order three automorphism $\gamma $ of $M$
which commutes with $\rho $ and such that $\gamma \alpha_g= \alpha_{\theta(g)} \gamma $, where $\theta$ is a cyclic permutation of the nontrivial elements of $\mathbb{Z}/2\mathbb{Z} \times \mathbb{Z}/2\mathbb{Z}   $. Conjugation by $\gamma  $ gives a tensor autoequivalence of $\mathcal{C} $ of order $3$, and hence determines an action of $\mathbb{Z}/3\mathbb{Z} $ on $\mathcal{C} $. 

The fusion rules for the equivariantization of $\mathcal{C} $ by this  $\mathbb{Z}/3\mathbb{Z} $-action were computed in \cite{1609.07604}. The simple objects are labeled $1$,
$\alpha$, $\alpha^2 $, $\beta $, $\xi $, $\alpha \xi $, $\alpha^2 \xi $, and $\beta \xi$.

The objects $\alpha $ and $\beta $ satisfy the fusion rules 
$$\alpha^3 \cong 1, \quad \beta^2 \cong 1 + \alpha + \alpha^2 +  2\beta, \quad \alpha \beta \cong  \beta \alpha \cong \beta $$ 
and the tensor subcategory they generate is equivalent to $\text{Rep}_{\mathcal{A}_4} $.

The other fusion rules are $$\xi \alpha \cong \alpha \xi, \quad \xi \beta  \cong \beta \xi, \quad \xi^2 \cong 1 + \xi + \beta \xi  .$$

Now let $\mathcal{D} $ be the even part of the $3^{\mathbb{Z}/4\mathbb{Z}}$ subfactor, again realized in $\text{End}_0(M)$. Let $P=M \rtimes_{\alpha_2} \mathbb{Z}/2\mathbb{Z}$, which is a von Neumann algebra generated by $M$ and a self-adjoint unitary $\lambda$ satisfying $\lambda x\lambda=\alpha_2(x) $ for all $x\in M $. Then we extend $ \alpha_g$ to an automorphism $ \tilde{\alpha}_g$ of $P $ by setting $\tilde{\alpha}_g (\lambda)=(-1)^g \lambda$, and $\rho $ to an endomorphism $\tilde{\rho} $ of $P$ by setting $\tilde{\rho}(\lambda)=\lambda $. 

Then we still have $\tilde{\alpha}_g\tilde{\alpha}_h= \tilde{\alpha}_{gh} $ and $\tilde{\alpha}_g \tilde{\rho} =\tilde{\rho} \tilde{\alpha}_{g^{-1}} $, but now $ \lambda \in (1=\alpha_0,\alpha_2)$, so there are only four distinct simple objects in the category, and we have the fusion rule
$$\tilde{\rho}^2 \cong 1+ 2\tilde{\rho}+ 2\tilde{\alpha}_1\tilde{\rho} .$$

\subsection{Subfactors at index $ 3+\sqrt{5}$}
There are exactly seven finite depth subfactors with index $ 3+\sqrt{5}$, up to isomorphism of the planar algebra and up to duality \cite{1509.00038}.

Three of these seven have intermediate subfactors, and are known as the first three ``fish'' subfactors, after the shape of their principal graphs. The principal even part of the $n^{th}$ fish subfactor is generated by objects $\alpha $ and $\rho $ satisfying the fusion rules $\alpha^2\cong 1 $, $\rho^2\cong 1 + \rho $, and $ (\alpha \rho)^n \cong (\rho \alpha)^n$.

The other four examples come from Haagerup-Izumi subfactors for order $4$ groups and their orbifolds.
The $ 3^G$ subfactors with $|G|=4$ have index $3+\sqrt{5} $. The $3^{\mathbb{Z}/2\mathbb{Z} \times \mathbb{Z}/2\mathbb{Z}  } $ subfactor is self-dual and the $3^{\mathbb{Z}/4\mathbb{Z}  }$ subfactor is not. These subfactors have also been constructed using planar algebra methods \cite{MR3314808,MR3402358}.

The $4442$ subfactor, first constructed using planar algebra methods in \cite{MR3314808}, has the principal graph given in Figure \ref{4442}. It is self-dual, and its even part is the 
$\mathbb{Z}/3\mathbb{Z} $-equivariantization of the even part of the $3^{\mathbb{Z}/2\mathbb{Z} \times \mathbb{Z}/2\mathbb{Z}  } $ subfactor described in the previous subsection.

\begin{figure}
\centerline{\includegraphics[width=1.25in]{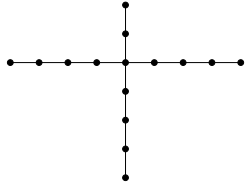}  }
\caption{Principal graph of the $4442$ subfactor}
\label{4442}
\end{figure}

The $2D2$ subfactor has principal graph given in Figure \ref{2d2}. Its even part is the  $\mathbb{Z}/2\mathbb{Z} $ de-equivariantization of the principal even part of the $3^{\mathbb{Z}/4\mathbb{Z}  }$ subfactor described in the previous subsection; it was also constructed using planar algebra methods in \cite{MR3394622}.

\begin{figure}
\centerline{ \includegraphics[width=1.2in]{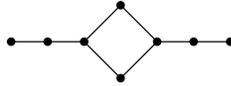}  .} 
\caption{Principal graph of the $2D2$ subfactor}
\label{2d2}
\end{figure}

Throughout this paper, $d$ will denote the number $2+\sqrt{5}$, which is the dimension of the noninvertible simple objects in the principal even parts of the $3^G $ subfactors for order four groups.

\section{The $3^{\mathbb{Z}/2\mathbb{Z} \times \mathbb{Z}/2\mathbb{Z}  }$ and  $4442$ subfactors}

\subsection{A first look at $3^{\mathbb{Z}/2\mathbb{Z} \times \mathbb{Z}/2\mathbb{Z}  }$}

Let $G=\mathbb{Z}/2\mathbb{Z} \times \mathbb{Z}/2\mathbb{Z} $. Let $\mathcal{C} $ be the even part of the $3^{G }$ subfactor, realized in $\text{End}_0(M)$ by automorphisms $\alpha_g$, $g \in G $, and an endomorphism $\rho $, as above. We also have the order three automorphism $\gamma$ of $M$, giving an automorphism of $\mathcal{C} $ by conjugation.  

There are (unique) Q-systems for $1+\alpha_g \rho$ for each $g \in G $, which all correspond to the $3^G$ subfactor. Since the $3^G$ subfactor for $\mathbb{Z}/2\mathbb{Z} \times \mathbb{Z}/2\mathbb{Z} $ is unique, $\text{Aut}(\mathcal{C}) $ must act transtively on these Q-systems, which are all fixed by $\text{Ad} \ \gamma $ and $\text{Inn}(\mathcal{C}) $. Therefore $\text{Out}(\mathcal{C}) $ has order at least $12$. It is shown in \cite{1609.07604} that $\text{Out}(\mathcal{C}) \cong \mathcal{A}_4 $, generated by $ \text{Ad} \ \gamma$ together with an automorphism which fixes the invertible objects but switches $\rho $ with another object.

We would like to classify the simple module categories over $\mathcal{C} $. We enumerate the fusion modules over the Grothendiek ring of $\mathcal{C} $ using the techniques of \cite{MR3449240} and a computer. There are $50$ (right) fusion modules, and their full data is contained in the accompanying text file \textit{Modules\_3\^{}\{Z2xZ2\}}. Here we just list in Figure \ref{calgs} the algebra objects that are associated to realizations of these modules by module categories, which is all we will need for our analysis.

\begin{figure}
\resizebox{0.9\hsize}{!}{
$
\begin{array}{c|cccc}
1 & \Gamma & \Gamma(1+4\rho ) &  \\
2 & \Gamma(1+\rho) & \Gamma (1+3\rho) & \\
3_g  (g \neq 0 ) & 1+\alpha_g \quad (\times 2) & 1+\alpha_g +2\Gamma \rho \quad (\times 2) &  \\ 
4_{g,h,k,l} (g,h \neq 0) & (1+\alpha_g) (1+\alpha_k\rho) \quad( \times 2) & (1+\alpha_h)(1+\alpha_l\rho) + \Gamma \quad (\times 2) & \\
5_g & 1+\alpha_g \rho \quad (\times 4) & \Gamma(1+3\rho)&  \\
6_g & 1+(\Gamma-\alpha_g)\rho \quad (\times 4) & \Gamma(1+\rho)&  \\
7 & 1\quad (\times 4) & 1+\Gamma\rho \quad  (\times 4) &  \\
\end{array} 
$
}
\caption{Algebra objects associated to realizations of the $50$ fusion modules over the Grothendieck ring of $\mathcal{C} $. \\\\
Here $\Gamma= \sum_{g \in \in G  } \limits \alpha_g$ and $\Gamma- \alpha_k = \sum_{g \in \in G, \ g \neq k  } \limits \alpha_g$. \\\\
The symbol ``$\times n $'' next to an object indicates the number of simple objects in the module category whose internal end is that object (the $n$ algebra structures may or may not be the same).\\\\
For $4_{g,h,k,l} $, all that matters is whether $k\in\{0,g\} $ and whether $l\in\{ 0,h \} $, so there are $3 \cdot 2 \cdot 3 \cdot 2 =36$ fusion modules of this type. }
\label{calgs}
\end{figure}

The category $\mathcal{C} $ contains $\text{Vec}_{G}$, so there are six module categories corresponding to the five subgroups of $G=\mathbb{Z}/2\mathbb{Z} \times \mathbb{Z}/2\mathbb{Z} $ (with two different module categories corresponding to the full subgroup $G$, since $H^2(G,\mathbb{C}^*)\cong \mathbb{Z}/2\mathbb{Z} $).  These six module categories provide realizations of the $5$ fusion modules $1$, $3_g$, and $7$ in Figure \ref{calgs}; these are all uniquely realized except for Module 1, which has two different realizations. Also, the unique Q-system for $1+\alpha_g \rho$ for each $g \in G$ gives a unique realization of each of the four fusion modules $5_g$. 

Thus we have already identified ten different module categories over $\mathcal{C} $, and we can get even more by composing the module categories corresponding to algebras which are sums of invertible objects with Morita autoequivalences coming from the algebras $1+\alpha_g \rho $.

\begin{lemma}
There are module categories over $ \mathcal{C}$ realizing each of the four fusion modules $6_g$.
\end{lemma}
\begin{proof}
Consider a Morita autoequivalence ${}_{\mathcal{C}} \mathcal{K} {}_{\mathcal{C} }$ realizing the fusion module $4_0$ (corresponding to the algebras $1+ \rho$ and $\Gamma(1+3\rho)  $). Then by Proposition \ref{division} we can divide by the $4$-dimensional subalgebra $\Gamma$ of $\Gamma(1+3\rho) $ to obtain invertible bimodule categories ${}_{\mathcal{C}} \mathcal{L} {}_{\mathcal{D} }$ and
${}_{\mathcal{D}} \mathcal{M} {}_{\mathcal{C} }$ which compose to $\mathcal{K} $, and such
that $\mathcal{ M}$ corresponds to a $4$-dimensional algebra in $\mathcal{C} $, and $\mathcal{L} $
corresponds to a $(3d+1)$-dimensional algebra in $\mathcal{C} $. Therefore $\mathcal{L}^{op} $ must realize one of the fusion modules $6_g$. Since $\text{Aut}(\mathcal{C}) $
acts transitively on the non-invertible simple objects, all $4$ fusion modules in $6_g$ must be realized.

\end{proof}

 We will see later that there are also module categories realizing Module $2$. In addition, some of the fusion modules $4_{g,h,k,l} $, which correspond to algebras of dimension $2+2d$, are realized as well. This can be seen by composing module categories corresponding to simple algebras of dimension $2$ with Morita autoequivalences corresponding to algebras of dimension $d+1$.

We now have a rather large list of module categories over $ \mathcal{C}$. However, with the exception of those coming from the Q-systems for $1+\alpha_g \rho$, which are known to be self-dual, it is not obvious what the dual categories of these module categories are. We shall see later that in fact these module categories all have dual categories equivalent to $\mathcal{C} $. 

However, because $ \mathcal{C}$ has a great deal of symmetry, it is difficult to work out the bimodule category structure directly. It is easier to make sense of the landscape by first studying the $\mathbb{Z}/3\mathbb{Z} $-equivariantization of $\mathcal{C} $.

\subsection{The $4442$ subfactor}

We now consider the equivariantization  $\mathcal{C}^{\mathbb{Z}/3\mathbb{Z}} $ of $\mathcal{C} $ coming from $\text{Ad} (\gamma)$, which we call $\mathcal{C}_1 $. 


Recall that the simple objects are labelled $1$,
$\alpha$, $\alpha^2 $, $\beta $, $\rho $, $\alpha \rho $, $\alpha^2 \xi $, and $\beta \xi $, where the tensor subcategory generated by the dimension $3 $ object $\beta $ is equivalent to $\text{Rep}_{\mathcal{A}_4} $.
%
%

We also have the crossed-product category $\mathcal{C} \rtimes  \mathbb{Z}/3\mathbb{Z} $ generated  by $\mathcal{C} $ and $\gamma $, which we call $\mathcal{C}_2 $. The subcategory $\mathcal{D}_2 $ of invertible objects of $\mathcal{C}_2 $ is equivalent to $\text{Vec}_{\mathcal{A}_4} $.

The categories $\mathcal{C}_1 $ and $\mathcal{C}_2 $ are Morita equivalent, with the Morita equivalence corresponding to the unique $3$-dimensional simple algebra in each of these categories.

\begin{figure} 
\begin{centering}
\begin{tikzpicture}
\node at (-0.2,2) {$\text{Rep}_{\mathcal{A}_4} $ };
\node at (0.7,2) {$\subset$ };
\node at (2,2 ) { $\mathcal{C}_1=\mathcal{C}^{\mathbb{Z}/3\mathbb{Z}}$};
\node at (7,2) {$\mathcal{C}_2=\mathcal{C} \rtimes \mathbb{Z}/3\mathbb{Z}$};
\node at (8.8,2) {$\supset $};
\node at (8.8,0) {$\supset $};
\node at (10.3,2) { $\text{Vec}_{\mathcal{A}_4} $};
\node at (10.3,0) { $\text{Vec}_{ \mathbb{Z}/2\mathbb{Z} \times \mathbb{Z}/2\mathbb{Z} }$};
\node at (7,1) {$\cup$};
\node at (10.3,1) {$\cup$};
\node at (7,0) {$\mathcal{C} $ };
\node at (4.4,2.5) {$\mathbb{Z}/3\mathbb{Z} $ };
\draw[<->,thick] (3.2,2) --(5.3,2);
\end{tikzpicture}
\caption{The categories $\mathcal{C}_1 $ and $\mathcal{C}_2 $}
\label{smallind}
\end{centering}
\end{figure}

We would like to find all fusion categories in the Morita equivalence class, and to classify simple module categories over these categories and invertible bimodule categories between them.

One way to construct module categories over $\mathcal{C}_1 $ and $\mathcal{C}_2 $ is to consider categories of modules over algebras in their group theoretical subcategories $\mathcal{D}_1 \cong  \text{Rep}_{\mathcal{A}_4}$ and $\mathcal{D}_2  \cong \text{Vec}_{\mathcal{A}_4}$, as well as in the (non-group-theoretical) subcategory $\mathcal{C} $ of $ \mathcal{C}_2$. 

We first consider algebras and module categories associated to $\mathcal{A}_4 $. The group 
$\mathcal{A}_4 $ has five conjugacy classes of subgroups: the trivial group, $\mathbb{Z}/2\mathbb{Z} $, $\mathbb{Z}/3\mathbb{Z} $, $\mathbb{Z}/2\mathbb{Z} \times \mathbb{Z}/2\mathbb{Z} $, and $\mathcal{A}_4   $ itself. The Schur multipliers of the first three are trivial, and of the last two are each $\mathbb{Z}/2\mathbb{Z} $. Therefore there are seven simple module categories each over
$\text{Rep}_{\mathcal{A}_4} $ and $\text{Vec}_{\mathcal{A}_4} $. The two $12$-dimensional simple algebras and the unique $3$-dimensional simple algebra in each category correspond to Morita equivalences between $\text{Rep}_{\mathcal{A}_4} $ and $\text{Vec}_{\mathcal{A}_4} $, and the two $4$-dimensional simple algebras in each category correspond to Morita autoquivalences. The outer automorphism group of each of these categories has order $4$, which is the order of the outer autmorphism group of $\mathcal{A}_4 $ multiplied by the order of the Schur multiplier of $\mathcal{A}_4 $. The Brauer-Picard group therefore has order $12$, and was shown to be the dihedral group $D_{12} $   
in \cite{MR3210925}. 

There is a third category $\mathcal{D}_3 $ in the Morita equivalence class, which can be realized as the category of bimodules over any of the $2 $-dimensional simple algebras in
$\text{Vec}_{\mathcal{A}_4} $, or as the category of bimodules over the unique $6$-dimensional
simple algebra in $\text{Rep}_{\mathcal{A}_4}$. Like $\mathcal{D}_2 $, the category $\mathcal{D}_3 $ is a  $\mathbb{Z}/3\mathbb{Z} $-graded extension of $\text{Vec}_{\mathbb{Z}/2\mathbb{Z} \times \mathbb{Z}/2\mathbb{Z} }$, with a single simple object of dimension $2$ in each of the nontrivial homogeneous components. The outer automorphism group of $\mathcal{D}_3 $ is necessarily isomorphic to the Brauer-Picard group $D_{12}$. The list of module categories over $\text{Vec}_{\mathcal{A}_4} $ can be summarized schematically by the picture in Figure \ref{veca4}.

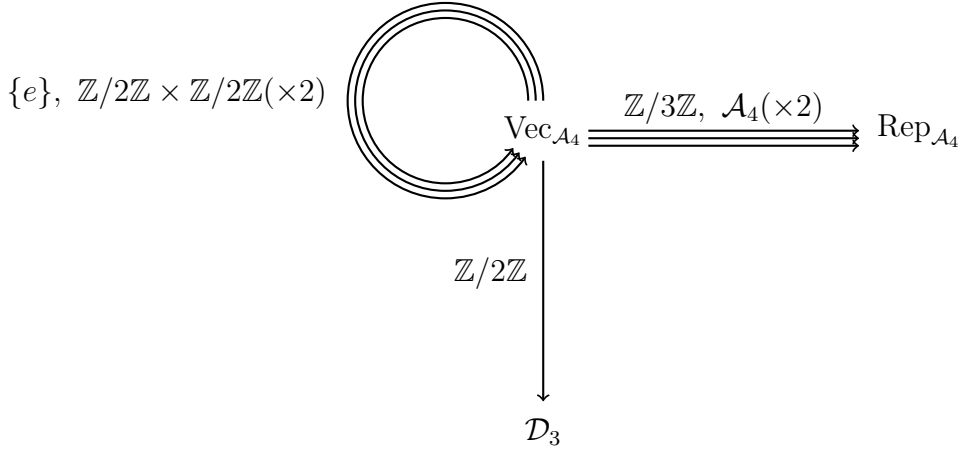
\begin{figure}
\begin{centering}
\begin{tikzpicture}

\draw[<-,thick] (0,0.4) --(0,3.6); 
\draw[->,thick] (0.6,3.9) --(4.2,3.9);
\draw[->,thick] (0.6,4) --(4.2,4);
\draw[->,thick] (0.6,3.8) --(4.2,3.8);
\node at (0,4) {$\text{Vec}_{\mathcal{A}_4}$};
\node at (0,0) {$\mathcal{D}_3$};
\node at (5,4) {$\text{Rep}_{\mathcal{A}_4}$};
\node at (2.4,4.3) {$ \mathbb{Z}/3\mathbb{Z} , \ \mathcal{A}_4 ( \times2)  $};
\node at (-0.7,2.1) {$ \mathbb{Z}/2\mathbb{Z}$};
\draw[->,thick] (0,4.4) arc (0:325:1.3); 
\draw[->,thick] (-0.1,4.4) arc (0:325:1.2); 
\draw[->,thick] (-0.2,4.4) arc (0:325:1.1); 
\node at (-5,4.5) {$\{e\},\ \mathbb{Z}/2\mathbb{Z} \times \mathbb{Z}/2\mathbb{Z} (\times 2) $};
\end{tikzpicture}
\caption{Module categories over $\text{Vec}_{\mathcal{A}_4} $ and their dual categories }
\label{veca4}
\end{centering}
\end{figure}

We now return to the categories $\mathcal{C}_1 $ and $\mathcal{C}_2 $.

\begin{lemma}
The group $\text{Out}(\mathcal{C}_1)$ is trivial.
\end{lemma}
\begin{proof}

The algebra in $\mathcal{C}_1 $ correponding to the $4442$ subfactor is $1+\xi $. The algebra structure on this object is unique by $4$-supertransitivity (which means that the principal graph begins as a sequence of $4$ consecutive edges before the first branch point), and there is no other self-dual simple object with the same dimension. Therefore the algebra is invariant under any automorphism of $\mathcal{C}_1 $.

Let $p$ be the minimal central idempotent at depth $5$ in the planar algebra of the $4442$ subfactor corresponding to the leg of length two emanating from the quadruple point of the principal graph in Figure \ref{4442}. Then $p$ generates the planar algebra, since there is a unique subfactor planar algebra with index $3+\sqrt{5} $ which is $4$-supertransitive but not $5$-supertransitive. Moreover, $p$ is necessarily fixed by any planar algebra automorphism. Therefore the planar algebra of the $4442$ subfactor has no nontrivial automorphisms. 

Since the algebra $1+\xi$ tensor generates $\mathcal{C}_1 $, is fixed by any automorphism of $\mathcal{C}_1 $, and there are no non-trivial automorphisms of the corresponding planar algebra, $\mathcal{C}_1$ does not have any nontrivial automorphisms either.
\end{proof}

There are 19 (right) fusion modules over the Grothedieck ring of $\mathcal{C}_1 $, which we again enumerate with a computer. The full data for these is contained in the accompanying text file \textit{Modules\_4442}. The chart in Figure \ref{c1algs} lists the algebra objects associated to realizations of each of these fusion modules. 
\begin{figure}
\resizebox{0.9\hsize}{!}{
$
\begin{array}{c|cccc}
1 & \Lambda +3\beta & (\Lambda +3 \beta)(1+4\xi) & & \\
2 & (\Lambda +3\beta)(1+\xi) & (\Lambda +3 \beta)(1+3\xi) & & \\
3 & \Lambda +\beta \ (\times 2) & \Lambda (1+2\xi)+\beta(1+6 \xi) \ (\times 2) & & \\
4 & (\Lambda +\beta)(1+\xi) \ (\times 2) & \Lambda (1+2\xi)+\beta(1+4 \xi) \ (\times 2)& & \\
5 & (\Lambda +\beta)(1+\xi) \ (\times 2) & \Lambda (1+\xi)+\beta(1+5 \xi) \ (\times 2)& & \\
6 & \Lambda +\beta+2\beta \xi \ (\times 2) & \Lambda (1+2\xi)+\beta(1+4 \xi) \ (\times 2)& & \\
7 & \Lambda +\beta+2\beta \xi \ (\times 2) & \Lambda (1+\xi)+\beta(1+5 \xi) \ (\times 2)& & \\
8 & \Lambda (1+\xi) \ (\times 4) & (\Lambda +3\beta)(1+3\xi) & & \\
9 & \Lambda (1+\xi) \ (\times 2) & \Lambda +\beta \xi \ (\times 2) & (\Lambda +3\beta)(1+3\xi)   & \\
10 & \Lambda +\beta \xi \ (\times 4) & (\Lambda +3\beta)(1+3\xi)  & & \\
11 &\Lambda (1+\xi) + 2 \beta \xi \ (\times 4) & (\Lambda +3\beta)(1+\xi) & & \\
12 &\Lambda (1+\xi) + 2 \beta \xi \ (\times 2) & \Lambda +3\beta \xi& (\Lambda +3\beta)(1+\xi)  & \\
13 &\Lambda  + 3 \beta \xi \ (\times 4) & (\Lambda +3\beta)(1+\xi) & &  \\
14 & 1+\beta \ (\times 3) & 1+(1+\Lambda )\xi+\beta(1+4\xi) \ (\times 3) & &\\
15 & (1+\beta)(1+\xi) \ (\times 3) & 1+\Lambda \xi+\beta(1+3\xi) \ (\times 3) & &\\
16 & 1+\xi \ (\times 3) & (\Lambda+2\beta)(1+\xi) & 1+\Lambda \xi+\beta(1+3\xi) \ (\times 3)  &\\
17 & 1+\beta \xi \ (\times 3) & (1+\beta)(1+\xi) \ (\times 3) & \Lambda (1+2\xi)+\beta(2+7\xi)  &\\
18 & 1 \ (\times 3) &  1+\xi+\beta\xi \ (\times 3)  & \Lambda (1+3\xi)+\beta(2+9\xi) & \Lambda +2\beta \\
19 & \Lambda   \ (\times 4)& \Lambda (1+\xi)+ 3\beta\xi \ (\times 4) & &
\end{array} 
$
}
\caption{Algebra objects associated to realizations of fusion modules over the Grothendieck ring of $\mathcal{C}_1 $. Here $\Lambda=1+\alpha+\alpha^2 $}
\label{c1algs}
\end{figure}

We will now determine which of these fusion modules are realized by module categories, and what the dual categories are in each case.

First we consider the seven simple algebras of  $\mathcal{D}_1 \cong \text{Rep}_{\mathcal{A}_4} $: one algebra structure each for the objects $1$, $\Lambda$, and $\Lambda+\beta$, and two each for $1+\beta $ and 
$\Lambda+3\beta$. Modules 3, 18, and 19 are realized uniquely by the categories of modules over the algebras $\Lambda+\beta$,  $1$, and $\Lambda$, respectively. 
\begin{lemma}\label{mod14}
Modules 1 and 14 are each realized by two distinct module categories.
\end{lemma}
\begin{proof}
From the table in Figure \ref{c1algs} we see that a realization of Module 1 has only one simple object with internal end $\Lambda+3\beta$, so the two algebra structures on $\Lambda+3\beta $ must give two different module categories. A realization of Module 14 has three distinct simple objects with internal end $1+\beta$. The category of modules over an algebra structure for $1+\beta$ in $\text{Rep}_{\mathcal{A}_4} $ also has three distinct simple objects with internal end $1+\beta $. The group of invertible objects in the dual category (which is also equivalent to $\text{Rep}_{\mathcal{A}_4} $) acts transitively on these three objects, so all three objects correspond to the same algebra structure on $1+\beta $. Therefore the group of invertible objects in the dual category of the full category of $(1+\beta)$-modules in $\mathcal{C}_1 $ also acts transitively on the set of simple objects whose internal end is $1+\beta$, and therefore these objects correspond to a single algebra structure on $1+\beta$. The two different algebra structures on $1+\beta $ thus correspond to two different module categories over $\mathcal{C}_1 $.
\end{proof}

\begin{lemma}
Module 16 is realized uniquely, and the dual category is again equivalent to $\mathcal{C}_1 $.
\end{lemma}
\begin{proof}
The unique algebra structure on $1+\xi$ corresponds to the $4442$ subfactor. Since the $4442$ subfactor is self-dual, $1+\xi$ corresponds to a Morita autoequivalence of  $\mathcal{C}_1$.
\end{proof}

\begin{lemma} \label{undual}
Module 17 is realized by the category of modules over a Q-system; the dual category of any realization of Module 17 coming from a Q-system is again $\mathcal{C}_1 $.
\end{lemma}
\begin{proof}
Since $1+\xi$ has a Q-system structure, so does the internal end of every simple object in the realization of Module 16.  The Q-system  $1+\Lambda\xi+\beta(1+3\xi)$ has dimension $4+12d$ and a subalgebra of dimension $4$. Since Module 16 gives a Morita autoequivalence, by Proposition \ref{division} there is a Q-system in $\mathcal{C}_1 $ with dimension $1+3d$, and that is only possible if Module 17 is realized. We compute the 
dual fusion ring of Module 17 using the methods of \cite[Section 3.1]{AHcat} and find that it must be the same as the Grothendieck ring of $\mathcal{C}_1 $. In particular there is a corresponding object $\xi' $ in the dual category $\mathcal{C}' $ satisfying $\xi'^2=1+\xi'+\beta'\xi' $. By \cite[Theorem 3.4]{ MR3449240}, $1+\xi' $ admits a Q-system structure, and by the uniqueness of the $4442$ subfactor, $\mathcal{C}' $ must be equivalent to $\mathcal{C}_1 $.
\end{proof}
\begin{corollary}
There are Morita autoequivalences $\mathcal{K}_1 $ and $\mathcal{K}_2 $ realizing Module 14
and a Morita autoequivalences $\mathcal{L} $  realizing Module 17
such that  $\mathcal{K}_1 \boxtimes_{\mathcal{C}_1} \mathcal{L}  $ and $\mathcal{L} \boxtimes_{\mathcal{C}_1} \mathcal{K}_2 $ each realize Module 16.
\end{corollary}
\begin{proof}
Let $\mathcal{L} $ be a Morita autoequivalence realizing Module 17, which we have just shown exists. Then $\mathcal{L} $ corresponds to an algebra structure for $(1+\beta)(1+\xi) $ in the copy of $\mathcal{C}_1 $ acting on the left, which necessarily has a subalgebra $1+\beta $. Therefore by Proposition \ref{division}, $\mathcal{L} $ decomposes as $\mathcal{K} \boxtimes_{\mathcal{F}} \mathcal{R}$, where $\mathcal{K} $ corresponds to an algebra of dimension $\text{dim}(1+\beta)=4$ in $\mathcal{C}_1 $, $ \mathcal{R}$ corresponds to an algebra of dimension $\text{dim}(1+\xi)=1+d $ in $\mathcal{C}_1 $, and $\mathcal{F} $ is the dual category with respect to either of these two algebras. Since there is a unique $(d+1)$-dimensional algebra in $\mathcal{C}_1 $ and it gives a Morita autoequivalence realizing Module 16, we have that $\mathcal{F}\cong \mathcal{C}_1 $ and $\mathcal{R} \cong \mathcal{K}^{op} \boxtimes_{\mathcal{C}_1} \mathcal{L}  $ realizes Module 16.  
Similarly, one can consider the right action of $\mathcal{C}_1 $ on $\mathcal{L} $ to obtain a factorization of $\mathcal{R} $ in the reverse order. 

\end{proof}

\begin{lemma}
Any Morita autoequivalence realizing Module 16 or Module 17 has order $2$ in the Brauer-Picard group.
\end{lemma}
\begin{proof}
Let $\mathcal{K} $ be a Morita autoequivalence realizing Module 16. Let $\kappa $
be an object in $\mathcal{K} $ whose left internal end is the algebra $1+\xi$. Then the right internal end of $\kappa $ must be $1+\xi$ as well. Since $(1+\xi,1+\xi)=2 $, by Frobenius reciprocity the object $\kappa \kappa $ in $ \mathcal{K} \boxtimes_{\mathcal{C}_1} \mathcal{K}$ has two simple summands and dimension $d+1$. Looking through the list of fusion modules in Figure \ref{c1algs}, we see that the only compatible possibility is Module 18, the trivial module.
Since $\mathcal{C} _1$ does not admit nontrivial outer automorphisms, $ \mathcal{K} \boxtimes_{\mathcal{C}_1} \mathcal{K}$ must be the trivial bimodule category. A similar argument works for Module 17, using the algebra $1+\beta \xi$.

\end{proof}
\begin{corollary}
Module 17 is realized uniquely.
\end{corollary}
\begin{proof}
We already know that Module 17 is realized by a Morita autoequivalence $\mathcal{K}$. 
Let $\mathcal{L} $ be another invertible bimodule category realizing Module 17. Then as in the proof of the Lemma,  $\mathcal{L} \boxtimes_{\mathcal{C}_1} \mathcal{K}$ must realize be the trivial bimodule category. Since $\mathcal{K} $ has order $2$ in the Braue-Picard group, this means that  $\mathcal{L}  \cong \mathcal{K}$.
\end{proof}

\begin{remark}
At this point we know that there are two different module categories realizing Module 14, and that there are Morita autoequivalences $\mathcal{K}_1 $ and $\mathcal{K}_2 $ realizing Module 14
such that  $$\mathcal{K}_1 \boxtimes_{\mathcal{C}_1} \mathcal{L}  \cong \mathcal{L} \boxtimes_{\mathcal{C}_1} \mathcal{K}_2 \cong \mathcal{M} ,$$
where $ \mathcal{L}$ and $\mathcal{M} $ are the unique Morita autoequivalences realizing Modules 17 and 16, respectively, and $\mathcal{K}_1 $ is the inverse of $\mathcal{K}_2 $ in the Brauer-Picard group. However we have not yet established whether $\mathcal{K}_1 $ is equivalent to $\mathcal{K}_2 $, and if so, whether the other module category realizing Module 14 is also a Morita autoequivalence.

\end{remark}

\begin{lemma}\label{prevlem}
Module 15 is realized.
\end{lemma}
\begin{proof}
Let $\mathcal{M} $ be the Morita autoequivalence realizing 
Module 16, and let $\mathcal{K}_1 $ and $\mathcal{K}_2 $ be distinct invertible bimodule categories realizing Module 14. 
Then each $ \mathcal{K}_i$ has a simple object whose (right) internal end in $\mathcal{C}_1 $ is the algebra 
$1+\beta$, and $\mathcal{M} $ has a simple object whose left internal end in $\mathcal{C}_1 $ is the algebra $1+\xi $. 

Then $\mathcal{K}_i \boxtimes_{\mathcal{C}_1} \mathcal{M} $ must correspond to an algebra with dimension $4(1+d)$ for each $i$, which is only possible if it realizes Module 17 or Module 15. Since Module 17 is realized by a unique module category, $\mathcal{K}_i \boxtimes_{\mathcal{C}_1} \mathcal{M} $ cannot realize Module 17 for both values of $i$. Therefore Module 15 must also be realized.
\end{proof}

\begin{lemma}
Both realizations of Module 14, and any realization of Module 15, are Morita autoequivalences.
\end{lemma}
\begin{proof}

Let ${}_{\mathcal{C}'} \mathcal{K} {}_{\mathcal{C}_1} $ be an invertible bimodule category realizing Module 14, and let $\mathcal{M} $ be the Morita autoequivalence realizing Module 16.
Then as in the proof of Lemma \ref{prevlem}, $ \mathcal{K} \boxtimes_{\mathcal{C}_1} \mathcal{M} $ contains a simple object whose internal end has dimension $4(1+d) $. From the fusion rules of $\mathcal{C}_1 $ we see that the internal end in $\mathcal{C}_1 $ necessarily has a subalgebra of dimension $4$. Then since all $4$-dimensional and $(d+1)$-dimensional algebras in $\mathcal{C}_1 $ are isomorphic to Q-systems,  by Proposition \ref{division} there must be a Q-system in $\mathcal{C}' $ with dimension $d+1$, giving a subfactor with index $ 3+\sqrt{5}$. Moreover, since $\mathcal{C}' $ is the category of $ \gamma$-$ \gamma$ bimodules with respect to a $4$-dimensional algebra in $\text{Rep}_{\mathcal{A}_4} \subset \mathcal{C}_1 $, it contains as a tensor subcategory the category of $\gamma $-$\gamma $-bimodules in $\text{Rep}_{\mathcal{A}_4} $, which is again equivalent to $\text{Rep}_{\mathcal{A}_4} $. From the classification of subfactors with index $3+\sqrt{5} $ in \cite{1509.00038}, this implies that $\mathcal{C}' $ is equivalent to $ \mathcal{C}_1$, as $\mathcal{C}_1 $ is the only even part of a subfactor with index $3+\sqrt{5} $ which contains $\text{Rep}_{\mathcal{A}_4} $.

Finally, any realization of Module 15 can be factored as a relative tensor product of invertible bimodule categories realizing Modules 14 and 16, as above, and so is also a Morita autoequivalence.

\end{proof}

\begin{lemma}
Together with the trivial bimodule category, the two Morita autoequivalences realizing Module 14 form a subgroup of the Brauer-Picard group isomorphic to $\mathbb{Z}/3\mathbb{Z} $.
\end{lemma}

\begin{proof}
Let $\mathcal{K} $ and $\mathcal{L} $ be the two realizations of Module 14. Then each contains a simple object of dimension $2$. By Proposition \ref{easycomp}, there are objects in each of $\mathcal{K} \boxtimes_{\mathcal{C}_1} \mathcal{K} $, 
$\mathcal{K} \boxtimes_{\mathcal{C}_1} \mathcal{L} $, $\mathcal{L} \boxtimes_{\mathcal{C}_1} \mathcal{K} $, and $\mathcal{L} \boxtimes_{\mathcal{C}_1} \mathcal{L} $ with $2$ simple summands and dimension $4$. From the list of fusion modules we see that the only possibilities for the fusion modules of these tensor products are Modules 14 and 18. Module 14 is only realized by the two Morita autoequivalences $\mathcal{K} $ and $\mathcal{L} $, and Module 18  is only realized by the trivial autoequivalence. Therefore the set consisting of  $\mathcal{K} $, $\mathcal{L} $, and the trivial bimodule category is closed under relative tensor product, and forms a subgroup of the Brauer-Picard group.
\end{proof}

For ease of notation we will now switch to lower case letters for elements of the Brauer-Picard group. Let $g$ be one of the Morita autoequivalences corresponding to Module 14, let $m$ be the autoequivalence corresponding to Module 16, and let $l$ be the autoequivalence corresponding to Module 17. Then we have already determined that $g^3=m^2=l^2=1$, and by replacing $g$ with $g^2$ if necessary, we may assume that $mg $ is an autoequivalence realizing Module 15, which we call $n$. 

\begin{lemma}
The Morita autoequivalence $n$ is the unique realization of Module 15, and we have $lg^2=n $ and $n^2=1 $. The bimodule categories $g$ and $m$ generate a subgroup isomorphic to the symmetric group $\mathcal{S}_3 $ inside the Brauer-Picard group.
\end{lemma}

\begin{proof}
From Figure \ref{c1algs} we see that any autoequivalence $ \mathcal{N}$ realizing Module 15 corresponds to algebras of dimensions $4+4d $ and $4+12 d $, which all have $4$-dimensional subalgebras. Therefore there are autoequivalences $\mathcal{K}_1 $ and $\mathcal{K}_2 $ corresponding to $4$-dimensional algebras such that $ \mathcal{N} \boxtimes_{\mathcal{C}_1} \mathcal{K}_1$
and  $ \mathcal{N} \boxtimes_{\mathcal{C}_1} \mathcal{K}_2$ realize Modules 16 and 17 respectively. Since 
$m$ and $l$ are the unique autoequivalences realizing Modules 16 and 17, and $g $ and $g^2$ are the only autoequivalnces realizing Module 14, and we have $mg=n$, we must have $lg^2=n $. If $n'$ is another autoequivalence realizing Module 15, then by the same argument we must have either $n'=mg=n$, or $n'=mg^2=ng=l $, which is impossible. Therefore $n$ is the unique autoequivalence realizing Module 15. Finally, since $n^{-1}$ is the opposite bimodule category of $n$, its simple objects have the same dimensions as those of $n$. Therefore by uniqueness of the realization of Module 15, $n^{-1}=n $. It is then easy to see that $\{1,g,g^2,l,m,n\}$ is closed under multiplication, and forms a subgroup of the Brauer-Picard group isomorphic to $ \mathcal{S}_3$.
\end{proof}

We have classified realizations of Modules 14, 15, 16, 17, and 18, which all are realized by Morita autoequivalences, uniquely except in the case of Module 14. We now turn to the rest of the list in Figure \ref{c1algs}.

Recall that Module 19 is realized by a unique module category, corresponding to the unique algebra structure on $\Lambda$. The dual category of the realization of Module 19 is $\mathcal{C}_2 $, by construction. 

\begin{lemma}
The outer automorphism group $\text{Out}(\mathcal{C}_2) $ is trivial.
\end{lemma}
\begin{proof}
By looking at the dimensions of simple objects in $\mathcal{C}_2 $ we see that every simple $3$-dimensional algebra in $\mathcal{C}_2 $ is necessarily in the subcategory $\text{Vec}_{\mathcal{A}_4} $. Since $\mathcal{A}_4 $ has a unique conjugacy class of order $3$ subgroups, there is a unique simple module category over $\mathcal{C}_2 $ corresponding to an algebra of dimension $3$. Since $\text{Out}(\mathcal{C}_1) $ is trivial, this means there is a unique invertible $\mathcal{C}_1 $-$\mathcal{C}_2 $ bimodule category corresponding to a $3$-dimensional algebra.
On the other hand, the outer automorphism group $\text{Out}(\mathcal{C}_2) $ acts faithfully on the set of invertible $\mathcal{C}_1 $-$\mathcal{C}_2 $ bimodule categories corresponding to algebras of dimension $3$. Therefore $\text{Out}(\mathcal{C}_2) $ must be trivial.

\end{proof}

\begin{lemma}
There are exactly two module categories realizing Module 1, whose dual categories are each 
equivalent to $\mathcal{C}_2 $.
\end{lemma}

\begin{proof}
We can compose the invertible $\mathcal{C}_2 $-$\mathcal{C}_1 $  bimodule category realizing Module 19 with the Morita autoequivalences $g$ and $g^2$ of $\mathcal{C}_1 $
to obtain bimodule categories corresponding to $12$-dimensional simple algebras. This gives two invertible bimodule categories realizing Module 1, which must also be distinct as $\mathcal{C}_1 $-module categories since $\text{Out}(\mathcal{C}_2) $ is trivial. Conversely, any bimodule category realizing Module 1 corresponds to a $12$-dimensional simple algebra with a $4$-dimensional subalgebra and must decompose as a tensor product of the unique bimodule category realizing Module 19 with either $g$ or $g^2$.
\end{proof}

\begin{lemma} \label{comparg}
 There is exactly one module category realizing Module 8, exactly one module category realizing Module 13, and at least one module category realizing Module 2. The dual categories of these three module categories are each equivalent to $\mathcal{C}_2 $. Modules 9 and 12 are not realized.
\end{lemma}

\begin{proof}
Consider the invertible $\mathcal{C}_2 $-$\mathcal{C}_1 $-bimodule category $a$ realizing Module 19 (and corresponding to an algebra of dimension $3$). Let $b_1=ag$ and $b_2=ag^2$, which realize Module 1 (and correspond to algebras of dimension $12$).  

Composing $a$ with the autoequivalence $m$ (corresponding to the $(d+1) $-dimensional algebra $1+ \xi $), we get a bimodule category $am $ corresponding to an algebra of dimension $3+3d$. The possibilities for a compatible fusion module are Modules 8, 9, and 10. 

We have $am \cdot m=a$. Modules 9 and 10 each correspond to an algebra structure for $\Lambda+\beta\xi$. Since $(\Lambda+\beta \xi,1+\xi)=1 $, if $am$ realized Module 9 or Module 10, $a$ would have to correspond to an algebra with dimension $(3+3d)(1+d )$. But from the data for Module 19 in Figure \ref{c1algs} we see that this is not the case. So $am$ realizes Module 8. 

Suppose $r$ is another module category realizing Module 8, which corresponds to an algebra structure for $\Lambda(1+\xi)$. Then since $(\Lambda(1+\xi),1+\xi) =2$,  by Proposition \ref{easycomp} $rm$ must contain two simple objects the sum of whose dimensions is $\sqrt{3}(1+d)$. From Figure \ref{c1algs} we see that this is only possible if $rm $ realizes Module 19. Since Module 19 is realized uniquely by $a$, it follows that $r=am$. 

Finally, if $r$ is any realization of Module 9, then $r$ corresponds to algebra structures for both  $\Lambda(1+\xi)$ and $\Lambda+\beta\xi $, so $rm$ must contain both two simple objects whose dimensions sum to $\sqrt{3}(1+d)$ and a single simple object with dimension $\sqrt{3}(1+d)$. From Figure \ref{c1algs} we see that this is impossible, so Module 19 is not realized.

A similar argument shows that Module 13 is realized uniquely by the bimodule $al$ and Module 12 is not realized.  

The bimodule $an$ corresponds to algebras of dimensions $12+12d$ and $12+36d$, and therefore realizes Module 2.

In each of the three cases we have constructed a $\mathcal{C}_1 $-module category by composing the invertible $\mathcal{C}_2 $-$\mathcal{C}_1 $-bimodule category $a$ with a Morita autoequivalence of $\mathcal{C}_1 $, so the dual category is equivalent to $\mathcal{C}_2 $ in each case.

\end{proof}
\begin{remark}
We have shown that there is a realization of Module 2 whose dual category is  $\mathcal{C}_2$. We will see later that there is also another realization of Module 2 with a different dual category.
\end{remark}

There is a unique algebra structure for the object $\Lambda+\beta $ in $\text{Rep}_{\mathcal{A}_4} \subset \mathcal C_1 $. Let $\mathcal{C}_3 $ be the dual category of $\mathcal{C}_1 $ with respect to the algebra $\Lambda+\beta $. Since $\mathcal{C}_1$ and $\mathcal{C}_2 $ are dual to each other with repect to the $12$-dimensional simple algebras in $\text{Rep}_{\mathcal{A}_4} $ and $\text{Vec}_{\mathcal{A}_4} $, the category $\mathcal{C}_3 $ is also the dual category of $\mathcal{C}_2 $ with respect to any of the $2$-dimensional simple algebras in $\mathcal{C}_2 $. 

Since $ \mathcal{C}_2$ is a $\mathbb{Z}/3\mathbb{Z} $-graded extension of $\mathcal{C} $, and the $2$-dimensional algebras in $\mathcal{C}_2 $ are contained in $\mathcal{C} $, the category $\mathcal{C}_3 $ is also $\mathbb{Z}/3\mathbb{Z} $-graded. 

In some of the following proofs, we will omit details of multiplicative compatibility checks and just sketch the outline, as they are similar to previous arguments.

\begin{lemma}
There is a exactly one module category realizing Module 3 and exactly one module category realizing Module 5. The dual categories of these module categories are each equivalent to $ \mathcal{C}_3$. \end{lemma}

\begin{proof}

Since  $\Lambda+\beta$ has a unique algebra structure, there is a unique module category $e$ realizing Module 3, whose dual category is $\mathcal{C}_3 $ by definition. Composing $e$ with either of the Morita autoequivalences $g$ or $g^2$ (which correspond to algebra structures for $1+\beta$) gives a module category which has two simple objects the sum of whose dimensions is $2\sqrt{ 6} $. The only compatible fusion module is Module 3, so we must have $e=eg=eg^2 $.
Since the Morita autoequivalence $m$ corresponds to the algebra $1+\xi$, the module category $em$ corresponds to a $ 6(1+d)$-dimensional algebra, so it realizes one of Modules 4, 5, 6, or 7. 
From the relation $e=em^2 $ we can rule Modules 6 and 7, which correspond to algebra structures for $\Lambda+\beta+2\beta\xi $, using a similar argument as in the proof of Lemma \ref{comparg}. From the relation $e=eg^2=eg^2m^2=emgm=eml $, we can rule out Module 4, which corresponds to an algebra structure for $\Lambda(1+\xi)+\beta(1+4\xi) $, again using a multiplicative compatibility constraint. Therefore $em$ realizes Module 5. Finally, if $r$ is any module category realizing Module 5, we can again show by multiplicative compatibility that $rm=e $ and therefore $r=em $.

\end{proof}
\begin{corollary}
We have $\text{Out}(\mathcal{C}_3 ) \cong \mathbb{Z}/3\mathbb{Z} $.
\end{corollary}
\begin{proof}
There is a unique module category $e$ realizing Module 3. On the other hand, if we extend $e$ to a $\mathcal{C}_3$-$
\mathcal{C}_1$-bimodule category, then $e $, $eg$, and $eg^2$ are three different $\mathcal{C}_3$-$
\mathcal{C}_1$-bimodule categories which all realize Module 3. Therefore the group $\text{Out}(\mathcal{C}_3) $ must be three times as big as the trivial group $\text{Out}(\mathcal{C}_1) $.
\end{proof}
Note that there are also three different $\mathcal{C}_3$-$
\mathcal{C}_1 $ bimodule categories realizing Module 5.

So far we have identified an order $6$ subgroup of the Brauer-Picard group and identified the (right) fusion modules associated to six invertible $\mathcal{C}_i$-$\mathcal{C}_1 $ bimodule categories for $i=1,2,3 $. We have also classified $\mathcal{C}_1$-module categories which realize Modules 1, 3, 5, 8, 9, 12, 13, 14, 15, 16, 17, 18, and 19.

\begin{lemma}
Modules 2, 4, 6, 7, 10, and 11 are not realized by any Morita autoequivalences.

\end{lemma}

\begin{proof}
Each of these fusion modules corresponds to an algebra of dimension $3+3d$, $6+6d $, or $12+12d$, which have subalgebras of dimension $3$, $6$, or $12$, respectively. 
Suppose there is a Morita autoequivalence corresponding to an algebra of dimension $3+3d$. 
Then since the dual category of $\mathcal{C}_1 $ with respect to the unique $3$-dimensional simple algebra is $\mathcal{C}_2 $, by Proposition \ref{division} there is an invertible $\mathcal{C}_2$-$\mathcal{C}_1 $ bimodule category corresponding to an algebra of dimension $d+1$. However, the only division algebra of dimension $d+1$ in $\mathcal{C}_1 $ is $1+\xi$, whose dual category is again $\mathcal{C}_1 $, so this is impossible. The $6+6d $ and $12+12d$ cases are similar.

\end{proof}

\begin{theorem}
The Brauer-Picard group of the $\mathcal{C}_i $ is $\mathcal{S}_3 $.
\end{theorem}

\begin{proof}
We have classified all Morita autoequivalences of $\mathcal{C}_1 $ and their compositions.
\end{proof}

We still have not classified realizations of Modules 2, 4, 6, 7, 10, and 11, but we will temporarily turn to module categories over $\mathcal{C}_2 $ and finish the classification of $\mathcal{C}_1 $-module categories later.

\begin{lemma} \label{1pd}
There are Morita autoequivalences of $\mathcal{C}_2 $ and $\mathcal{C}_3 $ corresponding to $(d+1)$-dimensional algebras.

\end{lemma}

\begin{proof}
Module 8 is realized by a $\mathcal{C}_2$-$
\mathcal{C}_1 $-bimodule category, and corresponds to a $(3+3d)$-dimensional algebra $\Lambda+\beta \xi $, which has a $3$-dimensional subalgebra $\Lambda$. Since the dual category of $\mathcal{C}_1 $ with respect to the algebra $\Lambda$ is $\mathcal{C}_2 $, by Proposition \ref{division} there must be a $(d+1) $-dimensional algebra in $\mathcal{C}_2 $ which gives a Morita autoequivalence. A similar argument using the $(6+6d)$-dimensional algebra corresponding to Module 5 works for $\mathcal{C}_3  $. 
\end{proof}
\begin{corollary}\label{notinv}
The category $\mathcal{C}_3 $ is a $\mathbb{Z}/3\mathbb{Z} $-graded extension of $ \mathcal{C}$, and neither of its nontrivial homogeneous components contains an invertible object.
\end{corollary}
\begin{proof}
The $(d+1)$-dimensional algebra in $\mathcal{C}_3 $ corresponds to an invertible $\mathcal{C}_3 $-$\mathcal{C}_1 $ bimodule category which is a composition of  the $\mathcal{C}_3 $-$\mathcal{C}_1 $ bimodule category corresponding to the Q-system $\Lambda+\beta $ in $\mathcal{C}_1 $ with the autoequivalence of $ \mathcal{C}_1$ corresponding to the Q-system $1+\xi $. Therefore this $(d+1)$-dimensional algebra is a Q-system with respect to the unitary structure that $\mathcal{C}_3 $ inherits from $\mathcal{C}_1 $. From the classification of subfactors with index $3+\sqrt{5}$, this $(d+1)$-dimensional Q-system must correspond to the unique $3^{\mathbb{Z}/2\mathbb{Z} \times \mathbb{Z}/2\mathbb{Z}}    $ subfactor, so the trivial component of $\mathcal{C}_3 $ with respect to the  $\mathbb{Z}/3\mathbb{Z} $-grading inherited from $ \mathcal{C}_2$ must be equivalent to $\mathcal{C} $.
Since there is  an invertible $\mathcal{C}_3 $-$\mathcal{C}_2 $-bimodule category corresponding to the simple $2$-dimensional algebras $1+\alpha_g$ for all $0\neq g \in 
\mathbb{Z}/2\mathbb{Z} \times \mathbb{Z}/2\mathbb{Z} $, and since $(1+\alpha_g,1+\alpha_h)=1 $ for $ g \neq h$, the nontrivial homogenous components of $\mathcal{C}_3 $ each have an object of dimension $2$, and therefore correspond as $\mathcal{C} $-module categories to the fusion Module 1 in Figure \ref{calgs}.  So neither of the nontrivial homogeneous components has any invertible objects.
\end{proof}

\begin{lemma} \label{z3sub}
There are two Morita autoequivalences of $\mathcal{C}_2 $ which correspond to $4$-dimensional algebras. These autoequivalences also correspond to algebras of dimension $4+16d$,
and they generate a subgroup isomorphic to $\mathbb{Z}/3\mathbb{Z} $ in the Brauer-Picard group.
 
\end{lemma}
\begin{proof}
There are two different invertible $\mathcal{C}_2-
\mathcal{C}_1$-bimodule categories realizing Module 1, corresponding to the two algebra structures for $\Lambda + 3\beta$ in $\mathcal{C}_1 $. Since the $3$-dimensional algebra $ \Lambda$ is a subalgebra with respect to both algebra structures, and the dual category of $\mathcal{C}_1 $ with respect to $\Lambda $ is $\mathcal{C}_2 $, by Proposition \ref{division} there must be two different Morita autoequivalences of $\mathcal{C}_2$ corresponding to $4$-dimensional algebras. By looking at the corresponding fusion module over the Grothendieck ring of  $\mathcal{C} $ in Figure \ref{calgs}, we see that each such autoequivalence also corresponds to an algebra of dimension $4+16d$.  We will call these autoequivalences $h_1 $ and $h_2$. 

There are exactly two $4$-dimensional simple algebras in $\mathcal{C}_2 $, which are both given by algebra structures on the sum of the invertible objects in $\text{Vec}_{\mathbb{Z}/2\mathbb{Z} \times \mathbb{Z}/2\mathbb{Z} }  \subset \mathcal{C} \subset \mathcal{C}_2$. In particular, if we denote the algebras by $\gamma $ and $\gamma' $, then $(\gamma,\gamma')=(\gamma,\gamma)=(\gamma',\gamma')=4 $. Therefore, by Frobenius reciprocity, $h_i h_j$ must contain an object $\sigma $ with $\text{dim}(\sigma)=4$ and $(\sigma,\sigma)=4 $, for any $i,j $. It follows that $\sigma $ decomposes as either four distinct $1$-dimensional objects, or as two copies of a $2 $-dimensional simple object.
If the former holds, then $h_ih_j$ has $1$-dimensional objects, which means it is the trivial autoequivalence. In the latter case, $h_i h_j $ is again a Morita autoequivalence corresponding to a $4$-dimensional algebra. Since there are only two $4$-dimensional simple algebras in $\mathcal{C}_2 $ and $\text{Out}(\mathcal{C}_2) $ is trivial, this means that $h_ih_j $ is either $h_1$ or $h_2$.
Therefore, $h_1$, $h_2$ and the trivial autoequivalence form a subgroup of the Brauer-Picard group.
\end{proof}

The Morita autoequivalence of $\mathcal{C}_2 $ corresponding to a $(d+1)$-dimensional algebra (Lemma \ref{1pd}) and the Morita autoequivalences corresponding to the $4$-dimensional algebras (Lemma \ref{z3sub}) generate the Brauer-Picard group. Since all of the $(d+1)$-dimensional algebras and $4$-dimensional algebras in $\mathcal{C}_2 $ are contained in the trivial component $ \mathcal{C}$, by Proposition \ref{lrends} all of the Morita autoequivalences of $\mathcal{C}_2 $ correspond to algebras in $ \mathcal{C}$.

\begin{lemma}\label{othaut}
There is a Morita autoequivalence of $\mathcal{C}_2 $ which corresponds to algebras in $\mathcal{C}$ of dimensions $1+d$ and $4+12d$;  another one which
corresponds to algebras of dimensions $1+3d$ and $4+4d $; and another one which corresponds to algebras of dimensions $4+4d$ and $4+12d$. 
\end{lemma}

\begin{proof}
By Lemma \ref{1pd} there is a Morita autoequivalence $ p$ corresponding to a $(d+1)$-dimensional algebra. By looking at the associated fusion module over the Grothendieck ring of $ \mathcal{C}$ in Figure \ref{calgs} we see that $p$ also has corresponds to a $(4+12d)$-dimensional algebra. Since the $(4+12d)$-dimensional algebra has a subalgebra of dimension $4$, which gives a Morita autoequivalence, we can write $p=qh $, where $q$ and $h$ are Morita autoequivalences corresponding to algebras of dimension $1+3d $ and $4$, respectively. Looking at the fusion modules corresponding to $(3d+1)$-dimensional algebras in $\mathcal{C} $, we see that $q$ must also correspond to an algebra of dimension $4+4d $. By Lemma \ref{z3sub}, $h^2 $ also corresponds to an algebra of dimension $4$, so $qh^2=ph $ must correspond to algebras of dimensions $4(1+d) $ and $4(1+3d)$.
 \end{proof}
Since the Brauer-Picard group is $ \mathcal{S}_3$, it must be that the three Morita autoequivalences of Lemma \ref{othaut} all have order $2$.

As we have seen previously, the algebras $1+\alpha_g \rho$, $g \in \mathbb{Z}/2\mathbb{Z} \times \mathbb{Z}/2\mathbb{Z} $, give four different module categories over $\mathcal{C} $. However, since  $\text{Ad} (\gamma) $ acts transitively $\{ \alpha_g \rho \}_{g \neq 0} $, the algebras
$1+\alpha_g\rho $ for $g \neq 0$  all give the same module category over $\mathcal{C}_2 $. 

 \begin{lemma}\label{l1}
 The dual category of $\mathcal{C}_2 $ with respect to the algebra $1+\rho $ is again $\mathcal{C}_2 $. The dual category of $ \mathcal{C}_2$ with respect to the algebra $1+\alpha_g \rho$, $g \neq 0 $ is not $\mathcal{C}_2 $.
 \end{lemma}
 
 \begin{proof}
 Let $r$ be the invertible bimodule category corresponding to the algebra $1+\alpha_g \rho$, $ g \neq 0$ in $\mathcal{C}_2 $. Then $r$ also corresponds to an algebra $1+\alpha_h \rho$ where $h \neq g, h\neq 0 $. Since $(1+\alpha_g\rho,1+\alpha_h\rho)=1 $, this implies that there must be a simple object of dimension $d+1$ in the trivial autoequivalence of the dual category of $\mathcal{C}_2 $. But $\mathcal{C}_2 $ does not have any simple objects with dimension $d+1$. Therefore the dual category cannot be $\mathcal{C}_2 $. 
 
 On the other hand, we know from Lemma \ref{1pd} that there is a $(d+1)$-dimensional algebra which gives a Morita autoequivalence of $\mathcal{C}_2 $. The only possibility for such an algebra is $1+\rho$. 
 \end{proof}
 
 \begin{corollary}
 The $(3d+1)$-dimensional division algebra which gives a Morita autoequivalence of $\mathcal{C}_2 $
 is $1+\sum_{g \in \mathbb{Z}/2\mathbb{Z} \times \mathbb{Z}/2\mathbb{Z}, \ g \neq 0} \limits 
 \alpha_g \rho$, which has a unique division algebra structure. 
 \end{corollary}
 
 \begin{proof}
 Let $\gamma$ be the $(3d+1)$-dimensional division algebra which gives a Morita autoequivalence of $\mathcal{C}_2 $. Suppose $(\gamma,\rho)=1 $. Then $(1+\rho,\gamma)=2 $, so there is a Morita autoequivalence which contains two simple objects the sum of whose dimensions is $\sqrt{(1+d)(1+3d)} $. But we have already determined the dimensions of all of the algebras corresponding to all of the Morita autoequivalneces of $\mathcal{C}_2 $ in Lemmas \ref{z3sub} and \ref{othaut}, and there is no such Morita autoequivalence. Therefore we must have  $(\gamma,\rho)=0 $ and $\gamma=1+\sum_{g \in \mathbb{Z}/2\mathbb{Z} \times \mathbb{Z}/2\mathbb{Z},  g \neq 0} \limits 
 \alpha_g \rho$.
 \end{proof}
 
 Since the Brauer-Picard group is  $\mathcal{S}_3$ and the outer automorphism group
of $\mathcal{C}_3 $ is  $\mathbb{Z}/3\mathbb{Z} $, there are two $\mathcal{C}_2 $ module categories whose dual categories are $\mathcal{C}_3 $. 

\begin{lemma}\label{l2}
One of the $\mathcal{C}_2 $-module categories whose dual categories are $\mathcal{C}_3 $ corresponds to algebras of dimensions $2$ and $2+8d$, and the other corresponds to algebras of dimensions $2+2d$ and $2+6d$. There module categories do not correspond to algebras of any other dimensions.
\end{lemma}
\begin{proof}
We have already seen that $\mathcal{C}_3 $ is the dual category with respect to algebras of dimension $2$ in $ \mathcal{C}_2$. By looking at the corresponding fusion module (one of the modules $3_g$ in Figure \ref{calgs}) over $\mathcal{C} \subset \mathcal{C}_2 $, we see that the module category over a $2$-dimensional algebra in $\mathcal{C}_2 $ also corresponds to algebras of dimension $2+8d$. Moreover, since $2$ and $2+8d$ are the only dimensions of division algebras in $ \mathcal{C}$ corresponding to this fusion module, by Proposition \ref{gralgs}, the dimension of every algebra corresponding to the module category over $\mathcal{C}_2 $ must be one of these numbers as well.

Composing this module category with the Morita autoequivalence of $ \mathcal{C}_2$ corresponding to a $(d+1)$-dimensional algebra gives a module 
category corresponding to an algebra of dimension $2+2d$. Again looking at the associated fusion module over $\mathcal{C}$ (one of the modules $4_{g,h,k,l} $ in Figure \ref{calgs}), we see that such a module category also corresponds to algebras of dimension $2+6d $, and that $2+2d$ and $2+6d$ are the only numbers that appear.
\end{proof}
\begin{corollary} \label{c2}
The dual category of $ \mathcal{C}_2$ with respect to the algebra $1+\alpha_g \rho$, $g \neq 0 $, is not $\mathcal{C}_3 $.
\end{corollary}

Let $\mathcal{C}_4 $ be the dual category of $ \mathcal{C}_2$ with respect to the algebra $1+\alpha_g \rho$, $g \neq 0 $.
\begin{lemma}
The fusion category $\mathcal{C}_4 $ is not equivalent to $\mathcal{C}_1 $, $\mathcal{C}_2 $, or $\mathcal{C}_3 $.
\end{lemma}
\begin{proof}
The category $\mathcal{C}_4 $ is not equivalent to $\mathcal{C}_2 $ by Lemma \ref{l1}, and it is not equivalent to $\mathcal{C}_3 $  by Corollary \ref{c2}. The only algebra of dimension $1+d$ in $\mathcal{C}_1 $ is $1+\xi$, which gives a Morita autoequivalence of $\mathcal{C}_1 $, so $\mathcal{C}_4 $ is not equivalent to $\mathcal{C}_1 $ either.
\end{proof}
\begin{lemma}\label{notinv2}
The category $\mathcal{C}_4 $ is a $\mathbb{Z}/3\mathbb{Z} $-graded extension of $\mathcal{C} $, and neither of its nontrivial components contain an invertible object. 
\end{lemma}
\begin{proof}
The proof is similar to that of Corollary \ref{notinv}. The category $\mathcal{C}_4 $ inherits a $\mathbb{Z}/3\mathbb{Z} $-grading from $ \mathcal{C}_2$, and by the uniqueness of the $3^{\mathbb{Z}/2\mathbb{Z}}$ subfactor, the trivial component is equivalent to $ \mathcal{C}$. The nontrivial homogeneous components each contain simple objects of dimension $d+1$. Looking at the list of fusion modules over the Grothendieck ring of $\mathcal{C} $ in Figure \ref{calgs}, we find that this means that each of these nontrivial homogeneous components must realize one of the modules $4_{g,h,k,l}$, and in particular has no invertible objects.

\end{proof}

By looking at Module 5$_g$ in Figure \ref{calgs}, we see that the $\mathcal{C}_2 $-module category corresponding to $1+\alpha_g \rho$ also corresponds to an algebra of dimension $4+12d$, which has a subalgebra of dimension $4$. Since both simple algebras of dimension $4$ give Morita autoequivalences of $\mathcal{C}_2 $, this implies that there is a  $\mathcal{C}_2 $-module category corresponding to an algebra of dimension $1+3d$ whose dual category is $\mathcal{C}_4 $.

We now return to the classification of module categories over $ \mathcal{C}_1$, referring again to the list of fusion modules in Figure \ref{c1algs}.

\begin{lemma}
There are invertible $\mathcal{C}_4$-$\mathcal{C}_1 $-bimodule categories realizing Modules 10 and 11. Module 10 is realized uniquely.
\end{lemma} 

\begin{proof}
Composing the invertible $\mathcal{C}_2$-$\mathcal{C}_1 $-bimodule category corresponding to the $3$-dimensional algebra $\Lambda \in \mathcal{C}_1$ with an invertible $\mathcal{C}_4$-$\mathcal{C}_2 $-bimodule category corresponding to the $(d+1)$-dimensional algebra $1+\alpha_g \rho \in \mathcal{C}_2$ gives an invertible $\mathcal{C}_4$-$\mathcal{C}_1 $-bimodule category corresponding to an algebra of dimension $3+3d $. This $ \mathcal{C}_1$-module category must realize one of Modules 8-10, and realizations of Modules 8 and 9 have already been classified in Lemma \ref{comparg}. Therefore Module 10 is realized. On the other hand, any invertible bimodule category realizing Module 10 corresponds to an algebra of dimension $3+3d$ which has a $3$-dimensional subalgebra, so it must be a composition of the unique invertible $\mathcal{C}_2$-$\mathcal{C}_1 $-bimodule category corresponding to the algebra $\Lambda$ with a $\mathcal{X}$-$\mathcal{C}_2 $-bimodule corresponding to a $(d+1)$-dimensional algebra. The only possibilities for the $\mathcal{X}$-$\mathcal{C}_2$ bimodule category are the Morita autoequivalence of $  \mathcal{C}_2$ corresponding to $1+\rho$ or a $\mathcal{C}_4$-$\mathcal{C}_2 $ bimodule category corresponding to $1+\alpha_g \rho $. Since there is no $\mathcal{C}_2$-$\mathcal{C}_1 $-bimodule category realizing Module 10, it must be the latter. 

In a similar way, composing the invertible $\mathcal{C}_2$-$\mathcal{C}_1 $-bimodule category corresponding to $\Lambda$ with an invertible $\mathcal{C}_4$-$\mathcal{C}_2 $-bimodule category corresponding to a $(3d+1)$-dimensional algebra, we find that Module 11 is realized.
\end{proof}

Let $e$ and $f$ be invertible $\mathcal{C}_4$-$\mathcal{C}_1$-bimodule categories realizing Modules 10 and 11, which correspond to algebra structures for $\Lambda +\beta \xi $ and $\Lambda(1+\xi)+2\beta \xi$, respectively. We continue to use the letters $g $, $m$, and $l$ as above to refer to Morita autoequivalences of $\mathcal{C}_1 $ which correspond to algebras with dimensions $4$, $1+d$, and $1+3d$, respectively.

\begin{lemma}
Modules 2, 4, 6, and 7 are realized by invertible $\mathcal{C}_4$-$\mathcal{C}_1 $-bimodule categories. Modules 4, 6, 7, and 11 are realized uniquely, and Module 2 is realized by exactly two module categories.
\end{lemma}

\begin{proof}
We use multiplicative comptability, omitting the details of the calculations, as they are similar to previous arguments. Using the fact that Module 10 is realized uniquely and that $em^2=e$, we find that $em$ realizes Module 6; and that Module 6 is realized uniquely. Similarly, using $ el^2=e$, we find that $el$ realizes Module 7; and that Module 7 is realized uniquely. Using the fact that all realizations of Modules 8, 9, and 10 have already been classified, we find that $elm$ must realize Module 2. 

From $fl^2=f$, we see that $fl$ realizes Module 6. Since Module 6 is realized uniquely, this implies that Module 11 is realized uniquely as well. Then using $fm^2=f$, we find that $fm$ realizes Module 4; and Module 4 is realized uniquely.

Finally, we have already found two realizations of Module 2, one with dual category $\mathcal{C}_2$
and one with dual category $\mathcal{C}_4 $. For any realization of Module 2, composition with $m$ 
must give a realization of one of Modules 1, 4, 5, 6, or 7. But we have already classified realizations of all of these fusion modules, as well as their compositions with $m$.

 \end{proof}
\begin{corollary}
The group $\text{Out}(\mathcal{C}_4) $ is trivial.
\end{corollary}
\begin{proof}
We have identified six different fusion modules which are realized by $\mathcal{C}_1 $-module categories whose dual catergories are $\mathcal{C}_4 $. Since the order of the Brauer-Picard group is six, these $\mathcal{C}_1 $-module categories exhaust the invertible $\mathcal{C}_4 $-$\mathcal{C}_1 $ bimodule categories. Since an outer automorphism of $\mathcal{C}_4 $ cannot fix any invertible $\mathcal{C}_4 $-$\mathcal{C}_1 $ bimodule category, and since the six such bimodule categories include a category realizing Module 2, which has a different dimension vector than the other five, there cannot be any outer automorphism of $\mathcal{C}_4 $. 
\end{proof}

This completes the classification of realizations of the fusion modules over the Grothendieck ring of $\mathcal{C}_1 $, and hence, of the simple module categories over $\mathcal{C}_1 $.

\begin{theorem}
There are exactly $20$ simple module categories over each of the $\mathcal{C}_i $.
\end{theorem}
\begin{proof}
Of the $19$ fusion modules over the Grothendieck of $\mathcal{C}_1 $, described in Figure \ref{c1algs}, we have seen that Modules 9 and 12 are not realized, Modules 1, 2, and 14 are each realized by exactly two module categories, and all of the other modules are realized by unique module categories.
\end{proof}

\begin{theorem}
There are exactly four fusion categories in the Morita equivalence class of $\mathcal{C}_1$, up to equivalence.
\end{theorem}

\begin{proof}
We have classified all of the module categories over $\mathcal{C}_1 $, and all of the dual categories
belong to $\{ \mathcal{C}_1, \mathcal{C}_2, \mathcal{C}_3, \mathcal{C}_4 \}$.
\end{proof}

We would now like to describe the module categories over $\mathcal{C}_{2-4} $, which are all $ \mathbb{Z}/3\mathbb{Z}$-graded extensions of $\mathcal{C} $;  and ultimately describe the module categories over $\mathcal{C} $.

\begin{lemma}\label{0grade}

All $\mathcal{C}_{2-4} $-module categories whose dual categories are among $\mathcal{C}_{2-4} $ correspond to algebras in $\mathcal{C} \subset \mathcal{C}_i $.
\begin{itemize}
\item If the dual category is $\mathcal{C}_2 $, then the three homogeneous components of the module category with respect to the $\mathbb{Z}/3\mathbb{Z} $-grading are equivalent as $\mathcal{C} $-module categories.
\item If the dual category is $\mathcal{C}_3 $ or $\mathcal{C}_4 $, then the three homogeneous components are mutually inequivalent as $ \mathcal{C}$-module categories.
\end{itemize}
\end{lemma}
\begin{proof}
The Brauer-Picard group of $ \mathcal{C}_2$ is generated by Morita autoequivalences corresponding to algebras in $\mathcal{C} $, so by Proposition \ref{lrends} all invertible $\mathcal{C}_2$-$\mathcal{C}_2 $ bimodule categories correspond to algebras in $\mathcal{C} $. Since $\mathcal{C}_3 $ is the category of $(1+\alpha_g)$-$(1+\alpha_g)$ bimodules in $\mathcal{C}_2 $ (for any $g \neq 0$) and $1+\alpha_g \in \mathcal{C}$, by Corollary \ref{lrends}, all invertible $\mathcal{C}_2$-$\mathcal{C}_3 $ bimodule categories correspond to algebras in $\mathcal{C} $. Similary,  $\mathcal{C}_3 $ is the category of $(1+\alpha_g \rho)$-$(1+\alpha_g \rho)$ bimodules in $\mathcal{C}_2 $ (again for any $g\neq 0$) and $1+\alpha_g \in \mathcal{C}$. The corresponding statements about module categories over $\mathcal{C}_3 $ and $\mathcal{C}_4 $ then follow in a similar way from Proposition \ref{lrends} and Corollary \ref{lrends2}.
The statements about equivalence of the homogeneous components follow from Proposition \ref{cinv}, using the fact that the category $\mathcal{C}_2 $ is a quasi-trivial graded extension of $\mathcal{C} $, while the categories $\mathcal{C}_3 $ and $\mathcal{C}_4 $ do not have any invertible objects except in their trivial components.
\end{proof}

\subsection{Module categories over the even part of the $3^{\mathbb{Z}/2\mathbb{Z}  \times \mathbb{Z}/2\mathbb{Z}  } $ subfactor}

We can now use our knowledge of the Brauer-Picard groupoid of the $4442$ subfactor to classify the module categories over $\mathcal{C} $, the even part of the Haagerup-Izumi subfactor for $\mathbb{Z}/2\mathbb{Z}  \times \mathbb{Z}/2\mathbb{Z}  $.
\begin{theorem}
There are exactly $30$ simple module categories over the even part $\mathcal{C} $ of the $3^{\mathbb{Z}/2\mathbb{Z}  \times \mathbb{Z}/2\mathbb{Z}  } $ subfactor. The dual category of each of these module categories is again $\mathcal{C} $. 
\end{theorem}
\begin{proof}
The six $\mathcal{C}_2 $-module categories whose dual categories are $\mathcal{C}_1 $ do not correspond to algebras in $\mathcal{C} $, since $\mathcal{C}_1 $ is not $\mathbb{Z}/3\mathbb{Z} $-graded. There are six module $\mathcal{C}_2 $-module categories whose dual categories are $\mathcal{C}_2 $, another six whose dual categories are $\mathcal{C}_4 $, and two whose dual categories are $\mathcal{C}_3 $ (since $\text{Out}(\mathcal{C}_3 ) \cong \mathbb{Z}/3\mathbb{Z} $). All of these $\mathcal{C}_2 $-module categories are also module categories over $ \mathcal{C}$, and by Lemma \ref{0grade}, they correspond to algebras in $\mathcal{C} $. Again by Lemma \ref{0grade}, the graded components of these module categories belong to $6 \cdot 1 + 6 \cdot 3 + 2 \cdot 3 =30$ different equivalence classes. Conversely, any module category over $\mathcal{C} $ corresponds to an algebra in $\mathcal{C} $ and is therefore realized as a homogeneous component of a  $\mathbb{Z}/3\mathbb{Z} $-graded module category over $\mathcal{C}_2 $. Finally, since the trivial components of $\mathcal{C}_{2-4}$ are all equivalent to $\mathcal{C} $, all of the dual categories of the $\mathcal{C} $-module categories are equivalent to $\mathcal{C} $ as well.
\end{proof}
\begin{corollary}
There are no other fusion categories in the Morita equivalence class of $\mathcal{C} $.
\end{corollary}
We would like to describe the $30$ simple module categories over  $\mathcal{C} $. We refer to the list of fusion modules over the Grothendieck ring of $\mathcal{C} $ in Figure \ref{calgs}.
\begin{lemma}
The $30$ simple module categories over $\mathcal{C} $ include:
\begin{itemize}
\item Two module categories realizing Module 1;
\item Four module categories realizing Module 2;
\item Unique module categories realizing Modules $3_g$, for $0\neq g  \in \mathbb{Z}/2\mathbb{Z} \times \mathbb{Z}/2\mathbb{Z}$;
\item Unique module categories realizing Modules $5_g$ and $6_g$ for all $g \in \mathbb{Z}/2\mathbb{Z} \times \mathbb{Z}/2\mathbb{Z}$;
\item A unique module category realizing Module 7;
\item Twelve module categories realizing modules of the form Module $4_{g,h,k,l} $.

\end{itemize}

\end{lemma}
\begin{proof}
We have already seen the Morita autoequivalences of  $\mathcal{C}_2 $ in Lemmas \ref{z3sub} and \ref{othaut}, which by Lemma \ref{0grade} all give $ \mathbb{Z}/3\mathbb{Z}$-graded module categories whose homogeneous components are mutually equivalent as $\mathcal{C} $-module categories. These Morita autoequivalences include module categories corresponding to the trivial algebra (realizing Module 7),  the two algebras of dimension $4$ (each realizing Module 1), and the algebra $1+\rho $ (realizing Module $5_0$). One of the other two Morita autoequivalences has homogeneous components realizing Modules $6_0$. The last one has homogeneous components realizing Module 2.

One of the $\mathcal{C}_2 $-module categories whose dual category is equivalent to $\mathcal{C}_3 $ corresponds to the algebras $1+\alpha_g $ for $g \neq 0 $. The homogeneous components are $\mathcal{C} $-module categories realizing Module $3_g$ for each $g \neq 0 $.  The other $\mathcal{C}_2 $-module category whose dual category is equivalent to $\mathcal{C}_3 $ corresponds to algebras of dimensions $2+2d $ and $2+6d $. The homogeneous components are $\mathcal{C} $-module categories realizing some of the Modules $4_{g,h,k,l} $ (possibly for different values of $g,h,k,l $).

There is a $\mathcal{C}_2 $-module category whose dual category is $\mathcal{C}_4 $ corresponding to the algebras $1+\alpha_g \rho  $ for $g \neq 0 $; the homogeneous components realize Modules $5_g$, $g \neq 0 $. Since the dual categories with respect to the algebras $1+\sum_{k \neq g} \limits \alpha_k \rho$ for $g\neq 0 $ are not among $\mathcal{C}_{1-3} $,  there must be another $\mathcal{C}_2 $-module category with dual category $\mathcal{C}_4 $ whose homogeneous components realize Modules $6_g$, $g \neq 0 $.  Then using similar multiplicative compatibility arguments as above, we find that there is  $\mathcal{C}_2 $-module category whose dual category is $\mathcal{C}_4 $ whose homogeneous components all realize Module 
2, and another three module categories whose homogeneous components realize a subset of Modules $ 4_{g,h,k,l}$.

\end{proof}

Thus to determine the fusion modules of all the module categories over $\mathcal{C} $, it remains to determine the subset of Modules $4_{g,h,k,l} $ which are realized by twelve $\mathcal{C} $-module categories. 

Recall that each Module $ 4_{g,h,k,l}$ corresponds to algebra structures for $(1+\alpha_g)(1+\alpha_k \rho) $ and $(1+\alpha_h)(1+\alpha_l\rho)+\Gamma $, where $0 \notin \{g,h\} $; and that there are 36 such fusion modules, determined by whether $k \in \{ 0,g\} $ and whether $l \in \{ 0,h\} $. The group  $\text{Out}(\mathcal{C}) $ acts on the Grothendieck ring of $\mathcal{C} $, and this determines an action on the set $\{ \text{Module }4_{g,h,k,l} \} $. The group  $\text{Out}(\mathcal{C}) $ is generated by the automorphism $\text{Ad}(  \gamma)$, which cyclically permutes the nontrivial invertible objects and fixes $ \rho$, and an automorphism which fixes the invertible objects and transposes $\rho $ with one of the other noninvertible simple objects. So the action of $\text{Out}(\mathcal{C}) $ on $\{ \text{Module }4_{g,h,k,l} \} $ has four orbits: two $6$-element orbits for $g=h $ (one orbit consisting of $\{\text{Module }4_{g,g,k,k} \} $ and the other consisting of  $\{\text{Module }4_{g,g,k,l} \} $ where exactly one of $\{k,l\} $ belongs to $\{ 0,g\} $) , and two $12 $-element orbits for $g \neq h $ (one generated by Module $4_{g,h,0,0} $ and the other by Module $4_{h,g,0,0} $ for a given $g \neq h $) .

\begin{lemma}
For a $\mathcal{C} $-module category realizing Module $4_{g,h,k,l} $, we have $g \neq h $.
\end{lemma} 
\begin{proof}
Suppose there is a $\mathcal{C} $-module category realizing Module $4_{g,h,k,l} $ with  $g=h$. Then 
the corresponding algebras $(1+\alpha_g)(1+\alpha_k \rho) $ (with dimension $2+2d$) and $(1+\alpha_h)(1+\alpha_l\rho)+\Gamma $ (with dimension $2+6d$) have a common $ 2$-dimensional subalgebra $1+\alpha_g $. Since the dual category of every $\mathcal{C} $-module category is again $\mathcal{C} $, this means is a $\mathcal{C} $-module category corresponding to algebras of dimensions  $1+d $ and $1+3d $. But there is no such module category.
\end{proof}
\begin{corollary}
The twelve $\mathcal{C} $-module categories corresponding to algebras of dimension $2+2d$ realize twelve different fusion modules from among Modules $4_{g,h,k,l} $, which form a single orbit under the action of $\text{Out}(\mathcal{C}) $.
\end{corollary}
\begin{remark}
We do not determine here which of the two $12$-element orbits is realized. Since these two orbits are transposed by an automorphism of the Grothendieck ring of $\mathcal{C} $, to distinguish them it is necessary to label the invertible objects of $ \mathcal{C}$ and fix a solution of Izumi's equations, as in \cite[Section 9.4]{1609.07604}. It is then an interesting problem to calculate which orbit is realized. 

This can also be thought of as a question about Morita equivalence of algebras. For each $g \neq 0$, there is a unique $h \neq 0,g $ such that any division algebra of the form $(1+\alpha_g )(1+\alpha_k\rho) $ has the same category of modules as a certain division algebra of the form $(1+\alpha_h)(1+\alpha_l\rho)+\Gamma $ for some $l$. Then the question is, for a given $g$, which of the other two noninvertible objects is $h$?

\end{remark}

The twelve $\mathcal{C} $-module categories realizing Modules $4_{g,h,k,l}$ are realized as homogenous components of four different $\mathcal{C}_2 $-module categories, with respect to their $\mathbb{Z}/3\mathbb{Z} $-gradings. The automorphism $\text{Ad} (\gamma) $ of $\mathcal{C}_2 $ acts on the set of pairs of objects of the form $(1+\alpha_g)(1+\alpha_k \rho) $ and $(1+\alpha_h)(1+\alpha_l\rho)+\Gamma $. This action preserves the properties $k \in \{ 0,g\} $ and $l \in \{ 0,h\} $. Since each of the $12$-element orbits of $\{ \text{Module }4_{g,h,k,l}  \} $ under the action of $\text{Out}(\mathcal{C}) $ contains elements with each of the four possibilties for whether $k \in \{0,g \} $ and whether $l\in \{0,h\} $, this means that each of the four $\mathcal{C}_2 $-module categories corresponding to $(2+2d)$-dimensional algebras realizes a different one of these four possibilities.
 
\begin{lemma}
\begin{enumerate}
\item For the $\mathcal{C}_2 $-module category corresponding to algebras of dimension $2+2d$ whose dual category is $\mathcal{C}_3$, the fusion modules $4_{g,h,k,l}$ realized by its homogeneous components satisfy $k \in \{ 0,g\} $ and  $l \notin  \{ 0,h\}$.

\item  For one of the three categories $\mathcal{C}_2 $-module categories corresponding to algebras of dimension $2+2d$ whose dual category is $\mathcal{C}_4 $, the fusion modules $4_{g,h,k,l} $ realized by its homogeneous components satify $k \notin \{0,g \}$ and $l \in \{0,h \} $; for another one $k \notin \{0,g \}  $ and $l \notin \{0,h \} $; and for the third one $k \in \{0,g \} $ and $l \in \{0,h \} $.
\end{enumerate}

\end{lemma}

\begin{proof}
Let $\gamma=(1+\alpha_g )(1+\alpha_k \rho) $ be a division algebra in $\mathcal{C}_2 $ whose dual category is $\mathcal{C}_3 $. Suppose $k \notin \{ 0,g \} $. Then since the algebra $1+\rho $ gives a Morita autoequivalence of $\mathcal{C}_2 $ and  $(\gamma,1+\rho)=1 $, there is also a division algebra of dimension $(1+d)(2+2d) $ whose dual category is $\mathcal{C}_3 $. But by Lemma \ref{l2}, there is no such algebra. Therefore $k \in \{ 0,g \} $. 
Similarly, let $\delta=(1+\alpha_h )(1+\alpha_l \rho) +\Gamma $ be a division algebra in $\mathcal{C}_2 $ whose dual category is $\mathcal{C}_3 $. Suppose $l \in \{ 0,h \} $. Then $(\delta,1+\rho)=3 $, so there is  an object in an invertible  $\mathcal{C}_2-\mathcal{C}_3 $ bimodule category with $3$ simple summands and dimension $ \sqrt{(2+6d)(1+d)}$. But by Lemma \ref{l2} there is no such object. Therefore $l \notin \{ 0,h\} $. 
This proves the first part. Since the four $\mathcal{C}_2 $-module categories corresponding to algebras of dimension $2+2d $ must realize the four possibilities for whether $k \in \{0,g \} $ and whether $l\in \{0,h\} $, the second part follows. 
\end{proof}

We can now list all the division algebras in $\mathcal{C} $.
\begin{theorem}
There are exactly $60$ division algebras in $\mathcal{C} $ up to isomorphism. They are given as follows:

\begin{itemize}
\item Unique algebra structures for the objects $1$ and $1+(\Gamma-\alpha_g)\rho$;
\item  For each $g \neq 0 $, unique algebra structures for the objects   $1+\alpha_g $ and $1+\alpha_g+2\Gamma \rho $;
\item For each $g$, unique algebra structures for $1+\alpha_g\rho$ and $1+(\Gamma-\alpha_g)\rho$;
\item Two algebra structures each for $ \Gamma$ and $\Gamma(1+4\rho) $;
\item For each $g \neq 0 $ and each $k$, two algebra structures each for $(1+\alpha_g)(1+\alpha_k \rho) $
and $(1+\alpha_g)(1+\alpha_k \rho)+\Gamma $;
\item Eight algebra structures each for $\Gamma(1+\rho) $ and $\Gamma(1+3\rho) $.
\end{itemize}

\end{theorem}

\begin{proof}
All division algebras occur as internal ends of simple objects in module categories, with two objects having the same internal end if they are in the same orbit under the action of the invertible objects in the dual category. Since in our case the dual category is always $\mathcal{C} $ which has four invertible objects, the number of simple objects in a module category having a given division algebra $\gamma$ as its internal end is exactly $4$ divided by the number of invertible subobjects of $\gamma $. The data can then be read off our classification of module categories over $\mathcal{C} $ and their associated fusion modules.
\end{proof}

\begin{theorem}
The Brauer-Picard group of $\mathcal{C} $ has order $360$.

\end{theorem}
\begin{proof}
There are $30$ simple module categories over $\mathcal{C} $, the dual category of each of them is again $\mathcal{C}$, and $|\text{Out}(\mathcal{C})|=12 $.
\end{proof}
It is difficult to directly compute multiplication in this group from the list of module categories and automorphisms. In joint work with F. Xu we will determine the structure of the Brauer-Picard group by analyzing its action on the Drinfeld center $\mathcal{Z}(\mathcal{C}) $ and using constructions from conformal field theory.

\subsection{The categories $\mathcal{C}_3 $ and $\mathcal{C}_4 $}

We have seen how the homogeneous components with respect to the $\mathbb{Z}/3\mathbb{Z} $-grading of the simple $\mathcal{C}_2 $-module categories whose dual categories are among $\mathcal{C}_{2-4} $ correspond to $\mathcal{C} $-module categories. In particular, there are $15 $ such $\mathcal{C}_2 $-module categories whose 45 graded components realize the 30 simple $\mathcal{C} $-module categories. By Lemma \ref{0grade}, a similar story can be told for the categories $\mathcal{C}_3 $ and $\mathcal{C}_4 $.

\begin{lemma}\label{lt}
The group $\text{Out}(\mathcal{C}_3) $ acts transitively on the non-trivial invertible objects of $\mathcal{C}$.
\end{lemma}
\begin{proof}
There is a unique simple $\mathcal{C}_2 $-module category corresponding to $2$-dimensional algebras, whose dual category is $\mathcal{C}_3 $. Since $| \text{Out}(\mathcal{C}_2)|=1$ and $|\text{Out}(\mathcal{C}_3) |=3$, this means that there are three different $\mathcal{C}_3 $-module categories in a single orbit of the action of $\text{Out}(\mathcal{C}_3)$ which correspond to $2$-dimensional algebras. Therefore $\text{Out}(\mathcal{C}_3)$ acts transitively on the three $2$-dimensional algebras $1+\alpha_g$ in $\mathcal{C}_3 $, and in particular acts transitively on the set of objects $\{\alpha_g \} _{g \neq 0} $.
\end{proof}

\begin{lemma}\label{pl}
There is an algebra of dimension $d+1$ which is fixed by $\text{Out}(\mathcal{C}_3) $ and whose dual category is again $\mathcal{C}_3 $. 

\end{lemma}
\begin{proof}
By Lemma \ref{l2}, there is an invertible $\mathcal{C}_{3}$-$\mathcal{C}_{2} $-bimodule category corresponding to algebras of dimension $2+2d$. The algebra of dimension $2+2d$ in $\mathcal{C}_2 $ necessarily has a subalgebra of dimension $2$, and since $\mathcal{C}_3 $ is the dual category of $\mathcal{C}_2 $ with respect to all of its $2$-dimensional simple algebras, by Proposition \ref{division} there must be a Morita autoequivalence of $\mathcal{C}_3 $ corresponding to an algebra $ \gamma$ of dimension $d+1$, which is necessarily in $\mathcal{C} \subset \mathcal{C}_3 $. 

The category $\gamma $-mod in $\mathcal{C}_3$ inherits the $ \mathbb{Z}/3\mathbb{Z}$-grading of $\mathcal{C}_3 $. Recall that each of the nontrivial homogeneous components of $\mathcal{C}_3$ contains a simple object of dimension $2$. Since $\text{dim}(\gamma)=d+1 $ and $(\gamma,\sigma)=1 $ for each of the $4$-dimensional algebras $\sigma \in \mathcal{C}$, it must be that one of the homogeneous components of $ \gamma$-mod contains a simple object of dimension $\sqrt{1+d} $ and the other two homoegeneous components each contain a simple object of dimension $2\sqrt{1+d} $, and correspond to algebras of dimension $ 4+4d$. This means that one of the components realizes a fusion module of the form $5_g$ from Figure \ref{calgs} and the other components each realize a module of the form $2$ or $6_h$ for some $h$. In particular, $\gamma $-mod only corresponds to one $(d+1)$-dimensional algebra.

Since $|\text{Out}(\mathcal{C}_3) |=3$ and the order of the Brauer-Picard group is 6, there are exactly two $\mathcal{C}_{3}$-module categories whose dual categories are equivalent to $\mathcal{C}_3 $, including the trivial module category. Therefore the algebra $\gamma $ must be fixed by $ \text{Out}(\mathcal{C}_3)$.
\end{proof}
\begin{corollary}
There are three $\mathcal{C}_3$-module categories corresponding to $(d+1)$-dimensional algebras whose dual categories are equivalent to $ \mathcal{C}_4$.
\end{corollary}
\begin{proof}
The category $ \mathcal{C}\subset \mathcal{C}_3$ has four different algebras of dimension $d+1$, and we have just seen that for exactly one of these the dual category is again $\mathcal{C}_3 $. We have also already described the dual categories of all of the $(d+1)$-dimensional algebras in $\mathcal{C}_1 $ and $\mathcal{C}_2 $. Therefore the dual categories with respect to each of the other three $(d+1)$-dimensional algebras in $\mathcal{C}_3 $ must be equivalent to $\mathcal{C}_4 $. Finally, as in the proof of the lemma the three algebras give mutually inequivalent module categories.
\end{proof}
\begin{lemma}\label{lfff}
Let $\gamma $ be the $(d+1)$-dimensional algebra in $\mathcal{C}_3 $ whose dual category is $ \mathcal{C}_3$. Then there is a $g$ such that the homogeneous components of $\gamma $-mod, thought of as module categories over $ \mathcal{C} \subset \mathcal{C}_3$, realize Modules $5_g$, $6_g$, and $2$ (from Figure \ref{calgs}).
\end{lemma}
\begin{proof}
We have seen in the proof of Lemma \ref{pl} that one of the homogeneous components of $\gamma $-mod realizes a module of the form $5_g$ and the other components realize a module of the form $2$ or $6_h$ for some $h$. Since Module $5_g$ corresponds to an algebra of dimension $4(1+3d)$, and since both algebras of dimension $4$ in $\mathcal{C}_3 $ correspond to the trivial module category, $\gamma $-mod corresponds to an algebra of dimension $1+3d$. By Lemmas \ref{lt} and \ref{pl}, the group $ \text{Out}(\mathcal{C}_3)$ fixes $\alpha_g \rho $ and acts transitively on $\{ \alpha_h\rho \}_{h \neq g} $. Therefore the dual categories of the three algebras $1+(\Gamma-\alpha_h) \rho $,  $h \neq g $, must be the same. Therefore it must be that one of the homogeneous components of $\gamma $-mod realizes Module $6_g$, and the last one must then realize Module 2.
\end{proof}
 
 We now turn to $\mathcal{C}_4 $. Recall that the dimensions of the simple objects in the nontrivial homogeneous components of $\mathcal{C}_4 $ are all $d+1$ or $d-1$.

\begin{lemma}
The six simple algebras in $\text{Vec}_{\mathbb{Z}/2\mathbb{Z} \times \mathbb{Z}/2\mathbb{Z}} \subset \mathcal{C} \subset \mathcal{C}_4 $ give mutually inequivalent module categories whose dual categories are all equivalent to $\mathcal{C}_4 $.
 \end{lemma}
 
 \begin{proof}
 We have already determined the dual categories of all algebras of integer dimension in $\mathcal{C}_{1-3} $, none of which are equivalent to $\mathcal{C}_4 $. Therefore the dual categories of the simple algebras with integer dimension in $\mathcal{C}_4 $ must all be equivalent to $\mathcal{C}_4 $. These six algebras give mutually inequivalent $\mathcal{C}$-module categories. If $\gamma $ is an $n$-dimensional algebra in $ \mathcal{C} \subset \mathcal{C}_4$, then as in the proof of Lemma \ref{pl}, two of the homogeneous components of $\gamma $-mod must have objects of dimension $\sqrt{n}(d-1) $. From the list of fusion modules in Figure \ref{calgs}, we see that this implies that there cannot be two different homogeneous components of $\gamma $-mod corresponding to algebras of integer dimension. Therefore the six module categories are mutually inequivalent.
   \end{proof}
   
   We have seen earlier that there is an invertible $\mathcal{C}_2 $-$\mathcal{C}_4 $-bimodule category and three $\mathcal{C}_3 $-$\mathcal{C}_4 $-bimodule categories which correspond to $(d+1)$-dimensional algebras. As $\mathcal{C} \subset  \mathcal{C}_4 $ has four different  $(d+1)$-dimensional algebras, we would like to decide how their dual categories split up.
   \begin{lemma}
   For three of the $(d+1)$-dimensional algebras in $\mathcal{C} \subset \mathcal{C}_4 $, the dual category is $\mathcal{C}_2$; for the other one it is $\mathcal{C}_3 $.
   \end{lemma}
   \begin{proof}
   Suppose there are two different $(d+1)$-dimensional algebras in $\mathcal{C} \subset \mathcal{C}_4 $ whose dual category is $\mathcal{C}_i $. Since $(1+\alpha_g \rho,1+\alpha_h \rho) =1$, this implies that there is a simple object in a Morita autoequivalence of $\mathcal{C}_i $ with dimension $d+1$. But $\mathcal{C}_3 $ does not have a simple object with dimension $d+1$, and by Lemma \ref{lfff}, neither does the nontrivial module category whose dual category is $\mathcal{C}_3 $. Therefore $\mathcal{C}_i $ must be $\mathcal{C}_2 $, and there is a unique $(d+1)$-dimensional algebra in $\mathcal{C}_4 $ whose dual category is $\mathcal{C}_3 $.
   \end{proof}

Putting all this together, we can compare the way the various low-dimensional algebras in $\mathcal{C} $ behave inside $\mathcal{C}_{2-4} $; this is summarized in Figure \ref{clgs}.

\begin{figure}[!htb]
\resizebox{0.9\hsize}{!}{
$
\begin{array}{c|c|c|c|c|c|c|c}
\text{dim} &  \text{\# of algebras} &  \text{\# of }  \mathcal{C}_2\text{-modules} & \text{duals} & \text{\# of }   \mathcal{C}_3\text{-modules}   & \text{duals} & \text{\# of }  \mathcal{C}_4\text{-modules}  & \text{duals} \\
\hline
2 & 3 & 1  & \mathcal{C}_3 & 3 & \mathcal{C}_2;  \mathcal{C}_2;  \mathcal{C}_2 & 3 &   \mathcal{C}_4; \mathcal{C}_4; \mathcal{C}_4 \\ 
4 & 2 & 2 & \mathcal{C}_2;\mathcal{C}_2 & 1 & \mathcal{C}_3 & 2 & \mathcal{C}_4;\mathcal{C}_4 \\ 
d+1 & 4 & 2 \ (3/1)  & \mathcal{C}_4; \mathcal{C}_2 & 4 & \mathcal{C}_4;  \mathcal{C}_4;  \mathcal{C}_4; \mathcal{C}_3  & 2 \ (3/1) & \mathcal{C}_2;\mathcal{C}_3  \\ 

\end{array} 
$
}
\caption{For algebras of small dimension in  $\mathcal{C}\subset \mathcal{C}_{2-4}$, this table shows the number of corresponding module categories for each $\mathcal{C}_i $, as well as the dual categories of these module categories.\\\\
The ``(3/1)'' in the last row indicates that three of the $(d+1) $-dimensional algebras correspond to one of the module categories and the fourth algebra corresponds to the other module category.
}
\label{clgs}
\end{figure}
\FloatBarrier

\section{The $3^{\mathbb{Z}/4\mathbb{Z}} $ and $2D2$ subfactors}

As we shall see, the repesentation theoretic structure of the $3^{\mathbb{Z}/4\mathbb{Z}} $ subfactor is very different than that of the $3^{\mathbb{Z}/2\mathbb{Z} \times \mathbb{Z}/2\mathbb{Z}  } $ subfactor. 

\subsection{The $ 3^{\mathbb{Z}/4\mathbb{Z} }$ subfactor and its de-equivariantization}
Let $\mathcal{P}_1 $ be the principal even part of the $ 3^{\mathbb{Z}/4\mathbb{Z} }$ subfactor. Then there are eight simple objects in $\mathcal{P}_1$, labelled by $\alpha_g $ and $\alpha_g \rho $, for $g \in \mathbb{Z}/4\mathbb{Z} $, and satisfying the Haagerup-Izumi fusion rules.

We recall some details from the de-equivariantization construction in \cite{1609.07604}. Consider the Cuntz algebra $\mathcal{O}_5 $, generated by $S$ and $T_g$, $g \in \mathbb{Z}/4\mathbb{Z}$. Then $\alpha_g  $ and $\rho $ are defined by the formulas on \cite[page 29]{1609.07604}, using the structure constants from the solution on \cite[page 58]{1609.07604}. Extended to the von Neumann algebra closure $M$ with respect to a certain state, this gives a realization of $\mathcal{P}_1 $, with $S \in (1,\rho^2)$ and  $T_g \in (\alpha_g\rho,\rho^2) $ for each $g$. 
Recall that the de-equivariantization is constructed on $P=M \rtimes_{\alpha_2} \mathbb{Z}/2\mathbb{Z}$ by extending $ \alpha_g$ and $\rho $ to $P $, setting $\tilde{\alpha}_g (\lambda)=(-1)^g \lambda$ and $\tilde{\rho}(\lambda)=\lambda $.  Let $\mathcal{Q}_1$ be the fusion category generated by $\tilde{\rho} $.

\begin{lemma}
There are exactly two different Q-systems each for $1+\tilde{\rho}$ and $1+\tilde{\alpha}_1\tilde{\rho} $, which are transposed by the inner automorphism of $\mathcal{Q}_1 $ given by conjugation by $\tilde{\alpha}_1$.
\end{lemma}
\begin{proof}
Equivalence classes of Q-systems for $1+\tilde{\rho}$ are given by isometries in $(\tilde{\rho},\tilde{\rho}^2)$ satisfying the equations in \cite[Lemma 3.5]{MR2418197}, modulo sign. The space $(\tilde{\rho},\tilde{\rho}^2)$ is spanned by $T_0 $ and $T_2 \lambda $, which each give Q-systems. A straightforward calculation shows that there there are no others.
The Q-system for $1+\tilde{\rho} $ given by $T_0 \in  (\tilde{\rho},\tilde{\rho}^2)$ is isomorphic to the Q-system for $1+\tilde{\alpha}_2\tilde{\rho} $ given by $T_0\lambda $ in $(\tilde{\alpha}_2\tilde{\rho}, (\tilde{\alpha}_2\tilde{\rho})^2=  \tilde{\rho}^2)$. The automorphism of $\mathcal{Q}_1 $ coming from conjugation by $\tilde{\alpha}_1 $ sends $ \tilde{\rho}$ to $\tilde{\alpha}_2 \tilde{\rho} $ and sends $T_0 $ to $T_2 $. Therefore this automorphism transposes the two equivalence classes of Q-systems for $1+\tilde{\rho} $.
 The case $1+\tilde{\alpha}_1\tilde{\rho} $ is similar.
\end{proof}

Let $\phi $ be the automorphism of $P $ defined by $\phi(x)=x $ for $x\in M $ and $ \phi(\lambda)=-\lambda $. Then $\phi$ commutes with $ \tilde{\alpha}_g$ for all $g$ as well as with $\tilde{\rho} $, so conjugation by $\phi $ gives a tensor autoequivalence $ \text{Ad}( \phi)$ of the category $\mathcal{Q}_1 $.

\begin{lemma}
The automorphism $ \text{Ad}( \phi) $ is outer.
\end{lemma}
\begin{proof}
Since $\text{Ad}(\phi) $ fixes both Q-systems for $1+\tilde{\rho} $, it is not isomorphic to conjugation by $ \tilde{\alpha}_1$. Therefore it suffices to show that it is not isomorphic to the identity. Suppose $\tau: Id \rightarrow \text{Ad}(\phi) $ is a monoidal natural isomorphism. Then $ \tau$ detemines a nonzero scalar $ \tau_g \in (\tilde{\alpha}_g,\tilde{\alpha}_g)$ for each $g$ and a nonzero scalar $\tau_{\rho} \in (\tilde{\rho},\tilde{\rho}) $. Since $\lambda \in (\tilde{\alpha}_1,\tilde{\alpha}_3) $ and $\phi(\lambda)=-1 $, by naturality we must have $\tau_1=-\tau_3 $. However since $\tilde{\alpha}_1\tilde{\rho}=\tilde{ \rho} \tilde{\alpha}_3$, by the monoidal property, we must have $\tau_1\tau_{\rho}=\tau_{\rho}\tau_3 $, which is a contradiction.
\end{proof}

\begin{lemma}
The category $\mathcal{P}_1$ is the equivariantization of $\mathcal{Q}_1 $ by the $\mathbb{Z}/2\mathbb{Z}$-action coming from $ \text{Ad} (\phi) $.
\end{lemma}

\begin{proof}
Each $\tilde{\alpha}_g $ has two different equivariant structures for $\text{Ad}(\phi) $, corresponding to the scalars $ \pm 1$ in $(\text{Ad}(\phi)(\tilde{\alpha}_g),\tilde{\alpha}_g)=(\tilde{\alpha}_g,\tilde{\alpha}_g)$. Denote these two equivariant structures by $ \tilde{\alpha}_g^{\pm}$. As an equivariant morphism, $\lambda \in ( \tilde{\alpha}_0^+,\tilde{\alpha}_2^-) $, but $\lambda \notin   ( \tilde{\alpha}_0^+,\tilde{\alpha}_2^+) $. The equivariantization is therefore generated by the $\tilde{\alpha}_g^+ $ along with  $\tilde{\rho}$ (with trivial equivariant structure), and it is easy to see that this is the same as the original category $\mathcal{P}_1 $.
\end{proof}

Let $\mathcal{Q}_2$ be the dual category of $ \mathcal{Q}_1$ with respect to a Q-system for $1+\tilde{\rho} $. The Grothendieck ring of $\mathcal{Q}_2 $ is determined by the dual graph of the $2D2$ subfactor, which was computed in \cite{MR3394622}. There are six simple objects in $\mathcal{Q}_2 $, with dimensions $1$, $\displaystyle \frac{d-1}{2} $, $\displaystyle \frac{d-1}{2} $
, $\displaystyle \frac{d+1}{2} $, $\displaystyle \frac{d+1}{2} $, and $d $. The two objects with dimension $\displaystyle \frac{d+1}{2} $ are dual to each other.

We can write down the fusion modules over the Grothendieck ring of $\mathcal{Q}_2 $. In addition to the trivial module and the module realized by the (dual) $2D2$ subfactor, there is one other fusion module, which has two simple objects with dimensions $d-1$ and $d+1$. The full data is in the accompanying text file \textit{Modules\_2D2\_dual}.

 This implies the following fact which we will need later.

\begin{lemma} \label{weirdlem}
Let $\mathcal{R} $ be a $\mathbb{Z}/2\mathbb{Z} $-graded extension of $ \mathcal{Q}_2$ which contains a division algebra $ \gamma$ of dimension $2+2d$ whose underlying object is a direct sum of four mutually non-isomorphic self-dual simple objects. Then $\mathcal{R} $ is a quasi-trivial extension and $\gamma $ has a subalgebra of dimension $2$.
\end{lemma}
\begin{proof}
Looking at the dimensions of the simple objects in the three fusion modules over the Grothendieck ring of $\mathcal{Q}_2 $, we find that the only possibility is that the nontrivial homogeneous component of $\mathcal{R} $ is the trivial $ \mathcal{Q}_2$-module category, and $\gamma$ is the sum of two invertible objects and two objects of dimension $d$. The sum of the two invertible objects is then a $2$-dimensional subalgebra.
\end{proof}

\subsection{The Brauer-Picard groupoid of $\mathcal{P}_1 $}

There are twelve fusion modules over the Grothendieck ring of $\mathcal{P}_1 $. As before we give the list of associated division algebras in Figure \ref{z4algs}; the full data is contained in the accompanying text file \textit{Modules\_3\^{}\{Z4\}}.

 \begin{figure}
$$
\begin{array}{c|ccc}
1 & \Pi & \Pi(1+4\rho) & \\
2 & \Pi(1+\rho) & \Pi(1+3\rho) &  \\
3 & \Phi \ (\times 2) & \Phi+2\Pi\rho \ (\times 2) &  \\
4 & \Phi(1+\rho) \ (\times 2) & \Phi(1+\rho)+\Pi\rho \ (\times 2)&  \\
5 & \Phi(1+\rho) \ (\times 2) & \Phi(1+\alpha_1 \rho)+\Pi\rho \ (\times 2)&  \\
6 & \Phi(1+\alpha_1 \rho) \ (\times 2) & \Phi(1+\rho)+\Pi\rho \ (\times 2)&  \\
7 & \Phi(1+\alpha_1 \rho) \ (\times 2) & \Phi(1+ \alpha_1 \rho)+\Pi\rho \ (\times 2)&  \\
8 & 1+\alpha_1 \rho \ (\times 2) & 1+ \alpha_3 \rho \ (\times 2)& \Pi(1+3\rho)   \\
9 & 1+\rho \ (\times 2) & 1+ \alpha_2 \rho \ (\times 2)& \Pi(1+3\rho)   \\
10 & 1+\rho +\Phi\alpha_1\rho  \ (\times 2) & 1+\alpha_2 \rho +\Phi\alpha_1\rho  \ (\times 2)  &\Pi(1+\rho) \\
11 & 1+\alpha_1 \rho +\Phi\rho  \ (\times 2) & 1+\alpha_3 \rho +\Phi\rho  \ (\times 2)  &\Pi(1+\rho) \\
12 & 1 \ (\times 4) &  1+\Pi\rho \ (\times 4)  &  \\
\end{array}
 $$
 \caption{Algebras associated to realizations of fusion modules over the Grothendieck ring of $\mathcal{P}_1$. Here $ \Pi=\sum_{g \in \mathbb{Z}_4} \limits \alpha_g $ and $\Phi =1+\alpha_2$.}
 \label{z4algs}
\end{figure}
\begin{lemma}

The outer automorphsim group $\text{Out}(\mathcal{P}_1) \cong \mathbb{Z}/2\mathbb{Z}$.
\end{lemma} 
\begin{proof}
Since there is a Q-system for $1+\alpha_g \rho$ for each $g$, by the uniqueness of the $ 3^{\mathbb{Z}/4\mathbb{Z}}$ subfactor, there must be an outer automorphism of $ \mathcal{P}_1$ taking $\rho $ to $\alpha_1 \rho $, so $\text{Out}(\mathcal{P}_1) $ is nontrivial. 

On the other hand, by \cite[Theorem 5.10]{1609.07604}, $\text{Out}(\mathcal{P}_1)$ is the stabilizer subgroup of the action on the gauge equivalence classes of solutions to Izumi's equations of  $(H^2(G, \mathbb{C}^*) \times G/2G) \rtimes \text{Aut}(G)$, for $G=\mathbb{Z}/4\mathbb{Z} $, which is isomorphic to $\mathbb{Z}/2\mathbb{Z}\times\mathbb{Z}/2\mathbb{Z}$. Since $\text{Aut}(G)$ does not fix the solution for $\mathbb{Z}/4\mathbb{Z} $ given in \cite[Section 9]{1609.07604}, $\text{Out}(\mathcal{P}_1)$ is not all of $\mathbb{Z}/2\mathbb{Z}\times\mathbb{Z}/2\mathbb{Z}$.
\end{proof}

\begin{lemma}
Modules 1, 3, 8, 9, and 12 are each realized uniquely.
\end{lemma}
\begin{proof}
For each of these fusion modules there is a corresponding algebra object with a unique algebra structure (namely, $\Pi $, $\Phi$, $1+\rho$, $1+\alpha_1\rho$, and $1$, respectively).
\end{proof}

Let $\mathcal{P}_2 $ be the dual category of $ \mathcal{P}_1$ with respect to the $(d+1)$-dimensional algebra $1+\rho$; let $\mathcal{P}_3 $ be the dual category of $ \mathcal{P}_1$ with respect to the $2$-dimensional algebra $\Phi$; and let $\mathcal{P}_4 $ be the dual category of $ \mathcal{P}_1$ with respect to the $4$-dimensional algebra $\Pi$.

The fusion rules for $\mathcal{P}_2 $ were computed in \cite{1308.5723}. The category $\mathcal{P}_2 $ contains two invertible objects, two objects with dimension $d$, and one object each with dimensions $d+1$ and $d-1$.

It is shown in \cite{1308.5723} that  $\mathcal{P}_2$ is the dual even part of the third fish subfactor. Let $\mathcal{P}_5 $ be the principal even part of the third fish subfactor.

Since $\mathcal{P}_1 $ is the $\mathbb{Z}/2\mathbb{Z} $-equivariantization of $\mathcal{Q}_1 $ by the action of $ \text{Ad}(\phi)$, the category $\mathcal{P}_3 $ is the $\mathbb{Z}/2\mathbb{Z} $-graded extension of $\mathcal{Q}_1 $ generated by $\phi $. 

\begin{lemma}
The Grothendieck ring of $\mathcal{P}_4 $ is isomorphic to that of $\mathcal{P}_1 $.
\end{lemma}
\begin{proof}
The  $\Pi$-$\Pi $-bimodules in $\text{Inv}( \mathcal{P}_1)  $ generate a tensor subcategory of $\mathcal{P}_4 $ which is isomorphic to $\text{Vec}_{\mathbb{Z}/4\mathbb{Z}}$. Let $\kappa $ be a right $\Pi $-module such that $\kappa \bar{\kappa} =\Pi$. Then by Frobenius reciprocity, $(\bar{\kappa} \rho \kappa, \bar{\kappa} \rho \kappa ) =(\rho \kappa \bar{\kappa}, \rho \kappa \bar{\kappa})=4$. Since $\text{dim}(\bar{\kappa} \rho \kappa)=4d $ and the global dimension is $\dim(\mathcal{P}_4)=\dim(\mathcal{P}_1) =4(1+d^2)$, this implies that either $\mathcal{P}_4 $ contains a single noninvertible simple object of dimension $2d$, or exactly four noninvertible simple objects, each with dimension $d$. Since Module 8 is realized, $\mathcal{P}_1 $ has a Q-system with dimension $4(1+3d)$ containing $\Pi $ as a subalgebra, so $\mathcal{P}_4 $ has a $(3d+1)$-dimensional algebra, and the latter holds. Let $\gamma $ be the $(3d+1)$-dimensional algebra in $ \mathcal{P}_4$. Then the dual category with respect to $\gamma $ is $\mathcal{P}_2 $. By lookinig at the dimensions of the simple objects in $\mathcal{P}_2 $, we see that dimension of the endomorphism space of any object of dimension $1+3d$ must be either $4$, $5$, or $7$. By Frobenius reciprocity, the dimension of the endomorphism space of $\gamma $ in $\mathcal{P}_4 $ must be the same, so the only possibility is $4$ and $ \gamma$ has $4$ distinct simple summands.  Since all of the objects in $\mathcal{P}_2 $ are self-dual, by  \cite[Lemma 3.6]{MR2914056}, all of the summands of $\gamma$ are self-dual. This implies that all four noninvertible simple objects of $\mathcal{P}_4 $ are self-dual. Then the principal graph for $\gamma $ can be computed by a similar argument as in \cite[Lemma 3.21]{MR2909758}, and that determines the Grothendieck ring.
 
\end{proof}

\begin{lemma} \label{dislem}
There is no division algebra of dimension $d+1$ in $\mathcal{P}_4 $; similarly there is no division algebra of dimension $3d+1$ in $\mathcal{P}_1 $. 
\end{lemma}

\begin{proof}
This follows from a similar argument as in the proof of \cite[Lemma 3.15]{MR2909758}, using the fact that there is no fusion module over the Grothendieck ring of $\mathcal{P}_2 $ corresponding to a division algebra of dimension $(1+d)(1+3d) $. (The list of fusion modules for the Grothendieck ring of $\mathcal{P}_2 $ is in the accompanying text file \textit{Modules\_3\^{}\{Z4\}\_dual}).
\end{proof}
\begin{corollary}
Module 2 is not realized by any module category over either $\mathcal{P}_1 $ or $\mathcal{P}_4 $.
\end{corollary}
\begin{proof}
A $ \mathcal{P}_1$-module category which realizes Module 2 would correspond to division algebras of dimension $4+4d$ and $4+12d$ in $\mathcal{P}_1 $, and therefore division algebras of dimension $1+d $ and $1+3d$ in $\mathcal{P}_4 $, which is impossible; and similarly  for  $ \mathcal{P}_4$-module categories.
\end{proof}

\begin{lemma}
The fusion categories $\mathcal{P}_{1-5} $ are mutually inequivalent.
\end{lemma}
\begin{proof}
All of these can be distinguished by their Grothendieck rings, except for  $\mathcal{P}_1 $ and $\mathcal{P}_4 $, which can be distinguished by Lemma \ref{dislem}.
\end{proof}

\begin{lemma} \label{2dlem}
The fusion categories $\mathcal{P}_2 $ and $\mathcal{P}_5 $ are dual to each other with respect to the unique $2$-dimensional simple algebras in each of them.
\end{lemma}
\begin{proof}
The principal graph of the third fish subfactor contains an odd vertex with Frobenius-Perron weight equal to $\sqrt{2} $. Therefore the corresponding $\mathcal{P}_2 $-$\mathcal{P}_5 $  bimodule category can be realized as (right) modules over a $2$-dimensional algebra in $\mathcal{P}_2 $ or as (left) modules over a $2$-dimensional algebra in $\mathcal{P}_5 $.
\end{proof}

\begin{lemma}\label{lemdu}
The dual category over any $\mathcal{P}_1 $-module category realizing one of the fusion modules 4,5,6, or 7, is $\mathcal{P}_5 $.
\end{lemma}

\begin{proof}
Let $\mathcal{K} $ be a $\mathcal{P}_1 $-module category realizing one of the fusion modules 4,5,6, or 7. Then $\mathcal{K} $ is equivalent to the category of modules over a $(2+2d) $-dimensional algebra in $\mathcal{P}_1 $. Such a $(2+2d) $-dimensional algebra necessarily has four mutually non-isomorphic self-dual simple summands. Therefore the dual algebra $\gamma $ in the dual category $(\mathcal{K}{}_{\mathcal{P}_1})^* $ also has dimension $2+2d$ and four mutually non-isomorphic self-dual simple summands. Since the algebras $\Phi(1+\alpha_g\rho) $ all have the $2$-dimensional algebra $\Phi $ as a subalgebra, and $\mathcal{P}_2 $ is the dual category of $\mathcal{P}_1 $ with respect to $\Phi $, the category $(\mathcal{K}{}_{\mathcal{P}_1})^* $ is the dual category of $ \mathcal{P}_2$ with respect to a $(d+1) $-dimensional algebra. Since $\mathcal{P}_2 $ is a $\mathbb{Z}/2\mathbb{Z} $-graded extensionn of $\mathcal{Q}_1 $ and the only $(d+1) $-dimensional algebras in $\mathcal{P}_2 $ are in $\mathcal{Q}_1 $, with dual category $\mathcal{Q}_2 $, the category   
$(\mathcal{K}{}_{\mathcal{P}_1})^* $ must be a $\mathbb{Z}/2\mathbb{Z} $-graded extension of $\mathcal{Q}_2 $. Therefore, by Lemma \ref{weirdlem}, $\gamma $ has a $2$-dimensional subalgebra. The dual category of this subalgebra must have a $(d+1)$-dimensional division algebra whose dual category is $\mathcal{P}_1 $. This means that the dual category of $(\mathcal{K}{}_{\mathcal{P}_1})^* $ over the $2$-dimensional subalgebra of $\gamma $ is $\mathcal{P}_2 $, so by Lemma \ref{2dlem}, it must be that $(\mathcal{K}{}_{\mathcal{P}_1})^* \cong \mathcal{P}_5$.
\end{proof}

\begin{theorem}
There are exactly five fusion categories in the Brauer-Picard groupoid of the $\mathcal{P}_i $, up to equivalence, and the Brauer-Picard group is $\mathbb{Z}/2\mathbb{Z} $.
\end{theorem}

\begin{proof}
The previous results show that the dual category of any module category over $\mathcal{P}_1 $ which realizes any of the $12$ fusion modules in Figure \ref{z4algs} must be among $\mathcal{P}_{1-5} $. Also, the only module category whose dual category is again $ \mathcal{P}_1$ is the trivial module category. Therefore the Brauer-Picard group is given by $\text{Out}(\mathcal{P}_1)\cong \mathbb{Z}/2\mathbb{Z} $.
\end{proof}

\begin{corollary}
The outer automorphism groups are  $\text{Out}(\mathcal{P}_1)\cong\text{Out}(\mathcal{P}_3)\cong\text{Out}(\mathcal{P}_4)\cong \mathbb{Z}/2\mathbb{Z} $, while $\text{Out}(\mathcal{P}_2)$ and $\text{Out}(\mathcal{P}_5)$ are trivial.
\end{corollary}

\begin{proof}
We already know that  $\text{Out}(\mathcal{P}_1)\cong \mathbb{Z}/2\mathbb{Z} $. Therefore both invertible $\mathcal{P}_1$-$\mathcal{P}_3$ bimodule categories are equivalent as $\mathcal{P}_1 $-module categories to the category of modules over the unique $2$-dimensional simple algebra $\Phi $ in $\mathcal{P}_1 $. Therefore   $\text{Out}(\mathcal{P}_3)$ must act transitively on this pair of bimodule categories, and hence is non-trivial. A similar argument holds for $\mathcal{P}_4 $, using the unique $4$-dimensional simple algebra $\Pi $ in $\mathcal{P}_1 $. The category $\mathcal{P}_2 $ is the dual category of $\mathcal{P}_1 $ with respect to a module category realizing Module 9. Since there is an automorphism of $\mathcal{P}_1 $ which sends $ \rho$ to $\alpha_1 \rho $, there is also a $\mathcal{P}_1 $-module category realizing Module 8, whose dual category is also $\mathcal{P}_2 $. Therefore the two invertible $\mathcal{P}_1$-$\mathcal{P}_2$ module categories are inequivalent as  $\mathcal{P}_1$-modules, so $\text{Out}(\mathcal{P}_2)$ must be trivial. A similar argument holds for $\mathcal{P}_5 $, using the fact that Modules 4, 5, 6, 7 each correspond to a division algebra whose simple summands include some but not all of the $\alpha_g \rho $. 
\end{proof}

\begin{lemma}
There are unique $\mathcal{P}_1 $-module categories realizing Modules 4 and 7, and none realizing Modules 5 or 6.
\end{lemma}

\begin{proof}
Let  $\mathcal{K} $ be an invertible $\mathcal{P}_5 $-$\mathcal{P}_1 $ bimodule category  realizing Module 5. Then there are objects $\kappa $ and $\lambda $ in $\mathcal{K} $ such that $\bar{\kappa}\kappa=(1+\alpha_2)(1+\rho) $ and $\bar{\lambda} \lambda=(1+\alpha_2)(1+\alpha_1 \rho) +\Pi \rho$. Then by Frobenius reciprocity, $(\lambda \bar{\kappa},\lambda \bar{\kappa})=(\bar{\kappa}\kappa,\bar{\lambda} \lambda )=4$. Since $\text{dim}(\lambda\bar{\kappa}) =\sqrt{(2+2d)(2+6d)}  $, there must either be a simple object in $ \mathcal{K} \boxtimes_{\mathcal{P}_2} \mathcal{K}^{op} \cong \mathcal{P}_5$ of dimension $\frac{1}{2} \sqrt{(2+2d)(2+6d)} $ or there must be four distinct simple objects in $\mathcal{P}_5 $ whose dimensions sum to $ \sqrt{(2+2d)(2+6d)}$. But neither of these possibilities hold. A similar argument rules out Module 6. Therefore any invertible $\mathcal{P}_5 $-$\mathcal{P}_2 $ bimodule category must realize either Module 4 or Module 7, and since these are transformed to each other by the outer automorphism of $\mathcal{P}_1 $, both fusion modules are realized unqiuely.
\end{proof}

\begin{remark}
It is shown in \cite{1609.07604} that the solution to Izumi's polynomial equations which give the category $\mathcal{P}_1 $ has an ``accompanying solution'' which gives another fusion category with the same fusion rules. This accompanying solution is the unique Haagerup-Izumi category for $\mathbb{Z}/4\mathbb{Z} $ which does not have a Q-system for $1+\rho $, so it must be the same as our $ \mathcal{P}_4$. From \cite[Theroem 5.10]{1609.07604}, we see that $\mathcal{P}_4 $ also has an automorphism which sends $\rho $ to $\alpha_1 \rho $.

We can then summarize the comparison between the representation theory of $\mathcal{P}_1 $ and $\mathcal{P}_4 $ as follows. They have the same Grothendieck ring and each admits seven simple module categories. Each has a unique module category realizing each of Modules 1, 3, and 12 (corresponding to the subgroups of $\mathbb{Z}/4\mathbb{Z} $; the dual category with respect to the $2$-dimensional algebra is $\mathcal{P}_3 $ in each case) and a unique module category realizing each of Modules 4 and 7 (whose dual categories are $\mathcal{P}_5 $). In addition, $\mathcal{P}_1 $ has module categories realzing Modules 8 and 9 (corresponding to $(d+1)$-dimensional algebras) and $\mathcal{P}_4 $ has module categories realizing Modules 10 and 11 (corresponding to $(3d+1)$-dimensional algebras); the duals to these module categories are $\mathcal{P}_2 $ in each case.   
\end{remark}

\begin{remark}
Since $\mathcal{P}_2 $ and $\mathcal{P}_5 $ have trivial outer automorphism groups, they must each admit a nontrivial module category which gives a Morita autoequivalence. It can be shown that in both cases, the nontrivial Morita autoequivalence has four simple objects, two each with dimension $d+1$ and $d-1$. In the case of $\mathcal{P}_5 $, this can be seen by a similar argument as in the proof of Lemma \ref{lemdu}, as follows. In Lemma \ref{lemdu} we identified $\mathcal{P}_5 $ using the property that it has six distinct simple objects whose dimensions sum to $4d$, allowing a composition of bimodule categories realizing Modules 4 and 7 with their opposites. On the other hand, the nontrivial $\mathcal{P}_5 $-$\mathcal{P}_5 $ bimodule category must admit a composition of a bimodule category realizing Module 4 with the opposite of a bimodule realizing Module 7, so it must have four distinct simple objects whose dimensions sum to $4d$. 
\end{remark}

We can now determine the module categories over the even parts $ \mathcal{Q}_1$ and $\mathcal{Q}_2 $ of the $2D2$ subfactor. 

We can write down the fusion modules over the Grothendieck ring of $\mathcal{Q}_1 $, of which there are seven. We do not need the details here, except to note that a fusion module which corresponds to a Q-system for $1+\tilde{\rho} $ does not also correspond to a Q-system for $1+\tilde{\alpha}_1 \tilde{\rho} $, so $1+\tilde{\rho} $ and $1+\tilde{\alpha}_1 \tilde{\rho} $ give inequivalent module categories over $\mathcal{Q}_1 $. The full data of the seven fusion modules is contained in the accompanying text file \textit{Modules\_2D2}.

\begin{theorem}
There are exactly three simple module categories over each of the $ \mathcal{Q}_i$.
\end{theorem}

\begin{proof}
Since $\mathcal{Q}_1 $ the trivial component of $ \mathcal{P}_3$ with respect to a $\mathbb{Z}/2\mathbb{Z} $-grading, every division algebra in $\mathcal{Q}_1$ gives a $\mathbb{Z}/2\mathbb{Z} $-graded module category over $\mathcal{P}_3 $. Therefore every $\mathcal{Q}_1 $-module category arises as a homogeneous component of a $\mathbb{Z}/2\mathbb{Z} $-graded module category over $\mathcal{P}_3 $. There are three module categories over $\mathcal{P}_3 $ which correspond to algebras in the trivial component: the trivial module category and the two module categories whose dual category is $\mathcal{P}_5 $, which correspond to Q-systems for $1+\tilde{\alpha}_g\tilde{\rho}$. 
Since $\mathcal{P}_5 $ is a quasi-trivial extension of $\mathcal{Q}_2$, by Proposition \ref{cinv} the two homogeneous components of each these module categories are equivalent at $\mathcal{Q}_1 $-module categories. Therefore there are exactly three simple $\mathcal{Q}_1 $-module categories, and hence three simple $\mathcal{Q}_2 $-module categories as well.

\end{proof}

\begin{theorem}
There are exactly two fusion categories in the Morita equivalence class of the $\mathcal{Q}_i $. The Brauer-Picard group of the $\mathcal{Q}_i $ is isomorphic to $\text{Out}(\mathcal{Q}_1) $, which has order $4$.
\end{theorem}
\begin{proof}
Since the three module categories over $\mathcal{Q}_1 $ include a unique module category whose dual category is again $\mathcal{Q}_1 $ and two module categories whose dual categories are $ \mathcal{Q}_2$, it must be that the Brauer-Picard group is isomorphic to $\text{Out}(\mathcal{Q}_1) $ and there are no other fusion categories in the Morita equivalence class besides $\mathcal{Q}_2 $.  

By the uniqueness of the $2D2$ subfactor, the automorphism group of $\mathcal{Q}_1 $ acts transitively on the four Q-systems of dimension $d+1$. The planar algebra of the $2D2$ subfactor is generated by the two minimal central idempotents corresponding to the two vertices past the first branch point of the principal graph (Figure \ref{2d2}), and any planar algebra automorphism must leave the set of these idempotents invariant. Therefore there is at most one nontrivial planar algebra automorphism, and hence at most one nontrivial automorphism of $\mathcal{Q}_1 $ which fixes a given Q-system of dimension $d+1$. Since $\phi $ is an outer automorphism which fixes all Q-systems of dimension $d+1$, there is exactly one such automorphism. Therefore $\text{Aut}(\mathcal{Q}_1) $ has order $8$. Since $\text{Ad} (\tilde{\alpha}_1) $ acts non-trivially, $ \text{Out}(\mathcal{Q}_1) $ has order $4$. \end{proof}

\begin{corollary}
There is a unique Q-system with dimension $d+1$ in $\mathcal{Q}_2 $ and $\text{Out}(\mathcal{Q}_2) \cong \mathbb{Z}/2\mathbb{Z}$.
\end{corollary}
\begin{proof}
Since $\mathcal{Q}_1 $ has a unique module category whose dual category is again $\mathcal{Q}_1 $, there is a also a unique module category over $\mathcal{Q}_2 $ whose dual category is $\mathcal{Q}_1 $. From the $2D2$ principal graph, we see that there are exactly two objects in this module category with dimension $\sqrt{d+1} $. These two objects are in the same orbit under the action of $\text{Inv}(\mathcal{Q}_1) $, so they correspond to the same Q-system in $\mathcal{Q}_2 $. Finally, since there are two $\mathcal{Q}_1 $-module categories whose dual category is $ \mathcal{Q}_2$ and the Brauer-Picard group has order four, it must be that  $\text{Out}(\mathcal{Q}_2)$ as order two. \end{proof}

It would be interesting to determine the multiplicative structure of the Brauer-Picard group. One can attempt to do so by explicitly describing the automorphisms of $\mathcal{Q}_1 $ using the frameork of \cite{1609.07604}, but we will not address this here.

The nontrivial Morita autoequivalence of $\mathcal{Q}_2 $ has two simple objects with dimensions $d-1$ and $d+1$.

\section{Other subfactors with index $3+\sqrt{5} $}
For the other subfactors with index $3+\sqrt{5}$, the Brauer-Picard groupoid is easy to
describe. 

\begin{enumerate}
\item It is observed in \cite{FXpaper} that the dual even part $\mathcal{F} $  of the second fish subfactor is the even part of the self-dual subfactor with principal graph $A_9$. Then $\mathcal{F} $ has a nontrivial Morita autoequivalence coming from the Q-system for the $A_9 $ subfactor. There is also a module category corresponding to the simple $2$-dimensional algebra in $\mathcal{F} $, which does not give a Morita autoequivalence (this module category also corresponds to the second fish subfactor). There is a fourth module category over $ \mathcal{F}$ coming from composition of this module category with the nontrivial Morita autoequivalence. There are no other module categories. There is a unique Q-system in $\mathcal{F}$ which corresponds to the $A_9$ subfactor, and the planar algebra of the $A_9$ subfactor is the Temperley-Lieb planar algebra, which has no nontrivial automorphisms; so $\mathcal{F} $ has no outer automorphisms either. Therefore there are exactly two fusion categories in the Morita equivalence class and the Brauer-Picard group is $\mathbb{Z}/2\mathbb{Z} $.
\item The even part $\mathcal{G} $ of the first fish subfactor is the tensor product of $\text{Vec}_{\mathbb{Z}/2\mathbb{Z}} $ with the even part of the subfactor with principal graph $A_4 $ (the latter is also known as the Fibonacci category). Then it is easy to see that $ \mathcal{G}$ does not admit any outer automorphisms, and $\mathcal{G} $ has a unique nontrivial module category which gives a Morita autoequvalence. Therefore the Brauer-Picard group is $\mathbb{Z}/2\mathbb{Z} $.

\end{enumerate}

\newcommand{\urlprefix}{}
\bibliographystyle{alpha}
\bibliography{bibliography}

\newcommand{\noopsort}[1]{}\def\cprime{$'$} \def\cprime{$'$} \def\cprime{$'$}
\begin{thebibliography}{DGNO10}

\bibitem[AH99]{MR1686551}
Marta Asaeda and Uffe Haagerup.
\newblock Exotic subfactors of finite depth with {J}ones indices
  {$(5+\sqrt{13})/2$} and {$(5+\sqrt{17})/2$}.
\newblock {\em Comm. Math. Phys.}, 202(1):1--63, 1999.
\newblock \mathscinet{MR1686551} \doi{10.1007/s002200050574}
  \arxiv{math.OA/9803044}.

\bibitem[AMP15]{1509.00038}
Narjess Afzaly, Scott Morrison, and David Penneys.
\newblock The classification of subfactors with index at most $5 \frac{1}{4}$,
  2015.

\bibitem[DGNO10]{MR2609644}
Vladimir Drinfeld, Shlomo Gelaki, Dmitri Nikshych, and Victor Ostrik.
\newblock On braided fusion categories. {I}.
\newblock {\em Selecta Math. (N.S.)}, 16(1):1--119, 2010.

\bibitem[EG11]{MR2837122}
David~E. Evans and Terry Gannon.
\newblock The exoticness and realisability of twisted {H}aagerup-{I}zumi
  modular data.
\newblock {\em Comm. Math. Phys.}, 307(2):463--512, 2011.

\bibitem[ENO05]{MR2183279}
Pavel Etingof, Dmitri Nikshych, and Viktor Ostrik.
\newblock On fusion categories.
\newblock {\em Ann. of Math. (2)}, 162(2):581--642, 2005.
\newblock \arxiv{math/0203060}, \mathscinet{MR2183279}.

\bibitem[ENO10]{MR2677836}
Pavel Etingof, Dmitri Nikshych, and Victor Ostrik.
\newblock Fusion categories and homotopy theory.
\newblock {\em Quantum Topol.}, 1(3):209--273, 2010.
\newblock \arxiv{0909.3140}, \mathscinet{MR2677836}.

\bibitem[GI08]{MR2418197}
Pinhas Grossman and Masaki Izumi.
\newblock Classification of noncommuting quadrilaterals of factors.
\newblock {\em Internat. J. Math.}, 19(5):557--643, 2008.
\newblock \arxiv{0704.1121}, \mathscinet{MR2418197}.

\bibitem[GI15]{1501.07679}
P.~{Grossman} and M.~{Izumi}.
\newblock {Quantum doubles of generalized Haagerup subfactors and their
  orbifolds}.
\newblock {\em ArXiv e-prints}, January 2015.
\newblock \arxiv{1501.07679}.

\bibitem[GIS15]{AHcat}
P.~Grossman, M.~Izumi, and N.~Snyder.
\newblock The {A}saeda-{H}aagerup fusion categories.
\newblock Preprint., 2015.

\bibitem[GP14]{1407.2783}
Cesar Galindo and Julia~Yael Plavnik.
\newblock Tensor functors between morita duals of fusion categories, 2014.

\bibitem[GS12]{MR2909758}
Pinhas Grossman and Noah Snyder.
\newblock Quantum subgroups of the {H}aagerup fusion categories.
\newblock {\em Comm. Math. Phys.}, 311(3):617--643, 2012.
\newblock \arxiv{1102.2631}, \mathscinet{MR2909758}.

\bibitem[GS16]{MR3449240}
Pinhas Grossman and Noah Snyder.
\newblock The {B}rauer-{P}icard group of the {A}saeda-{H}aagerup fusion
  categories.
\newblock {\em Trans. Amer. Math. Soc.}, 368(4):2289--2331, 2016.

\bibitem[IMP13]{1308.5723}
Masaki Izumi, Scott Morrison, and David Penneys.
\newblock Fusion categories between $c \boxtimes d$ and $c * d$, 2013.

\bibitem[Izu16]{1609.07604}
Masaki Izumi.
\newblock The classification of $3^n$ subfactors and related fusion categories,
  2016.

\bibitem[Jon]{math.QA/9909027}
Vaughan F.~R. Jones.
\newblock {Planar algebras, I}.
\newblock \arxiv{math.QA/9909027}.

\bibitem[Liu15]{MR3345186}
Zhengwei Liu.
\newblock Composed inclusions of {$A_3$} and {$A_4$} subfactors.
\newblock {\em Adv. Math.}, 279:307--371, 2015.

\bibitem[Lon94]{MR1257245}
Roberto Longo.
\newblock A duality for {H}opf algebras and for subfactors. {I}.
\newblock {\em Comm. Math. Phys.}, 159(1):133--150, 1994.
\newblock \mathscinet{MR1257245}.

\bibitem[LR97]{MR1444286}
R.~Longo and J.~E. Roberts.
\newblock A theory of dimension.
\newblock {\em $K$-Theory}, 11(2):103--159, 1997.
\newblock \mathscinet{MR1444286}.

\bibitem[MP15a]{MR3394622}
Scott Morrison and David Penneys.
\newblock 2-supertransitive subfactors at index {$3+\sqrt{5}$}.
\newblock {\em J. Funct. Anal.}, 269(9):2845--2870, 2015.

\bibitem[MP15b]{MR3314808}
Scott Morrison and David Penneys.
\newblock Constructing spoke subfactors using the jellyfish algorithm.
\newblock {\em Trans. Amer. Math. Soc.}, 367(5):3257--3298, 2015.

\bibitem[MS12]{MR2914056}
Scott Morrison and Noah Snyder.
\newblock Subfactors of index less than 5, {P}art 1: {T}he principal graph
  odometer.
\newblock {\em Comm. Math. Phys.}, 312(1):1--35, 2012.

\bibitem[NR14]{MR3210925}
Dmitri Nikshych and Brianna Riepel.
\newblock Categorical {L}agrangian {G}rassmannians and {B}rauer-{P}icard groups
  of pointed fusion categories.
\newblock {\em J. Algebra}, 411:191--214, 2014.

\bibitem[Ost03]{MR1976459}
Victor Ostrik.
\newblock Module categories, weak {H}opf algebras and modular invariants.
\newblock {\em Transform. Groups}, 8(2):177--206, 2003.
\newblock \mathscinet{MR1976459} \arxiv{0111139}.

\bibitem[Pop90]{MR1055708}
Sorin Popa.
\newblock Classification of subfactors: the reduction to commuting squares.
\newblock {\em Invent. Math.}, 101(1):19--43, 1990.
\newblock \mathscinet{MR1055708} \doi{10.1007/BF01231494}.

\bibitem[PP15]{MR3402358}
David Penneys and Emily Peters.
\newblock Calculating two-strand jellyfish relations.
\newblock {\em Pacific J. Math.}, 277(2):463--510, 2015.

\bibitem[Tam01]{MR1815142}
Daisuke Tambara.
\newblock Invariants and semi-direct products for finite group actions on
  tensor categories.
\newblock {\em J. Math. Soc. Japan}, 53(2):429--456, 2001.

\bibitem[Xu16]{FXpaper}
Feng Xu.
\newblock Examples of subactors from conformal field theory.
\newblock {\em Communications in Mathematical Physics (to appear)}, 2016.

\bibitem[Yam04]{MR2075605}
Shigeru Yamagami.
\newblock Frobenius algebras in tensor categories and bimodule extensions.
\newblock In {\em Galois theory, {H}opf algebras, and semiabelian categories},
  volume~43 of {\em Fields Inst. Commun.}, pages 551--570. Amer. Math. Soc.,
  Providence, RI, 2004.
\newblock \mathscinet{MR2075605}.

\end{thebibliography}

\end{document}